\documentclass[11pt]{article}

\usepackage[english]{babel}
\newtheorem{theorem}{Theorem}
\newtheorem{lemma}{Lemma}
\newtheorem{definition}{Definition}
\newtheorem{proof}{Proof}
\usepackage{stmaryrd}
\usepackage{enumerate}
\usepackage{algorithm}
\usepackage[noend]{algpseudocode}
\usepackage{calrsfs}
\usepackage{graphicx}
\usepackage{amssymb}
\usepackage{array}
\usepackage{url}
\usepackage{amsmath}
\usepackage{epstopdf}
\usepackage{epsfig}
\usepackage{subcaption}
\usepackage{fancyheadings}
\usepackage{pgfplots}
\usepackage{physics}
\usepackage{stackengine}
\usepackage[margin=1.0in]{geometry}
\DeclareMathAlphabet{\pazocal}{OMS}{zplm}{m}{n}

\usepackage{graphicx}

\newcommand{\lJump}{[\![}
\newcommand{\rJump}{]\!]}

\usepackage[framemethod=tikz]{mdframed}
\usepackage{lipsum}

\definecolor{mycolor}{rgb}{0.122, 0.435, 0.698}
\definecolor{blue-green}{rgb}{0.0, 0.87, 0.87}
\definecolor{royalblue}{rgb}{0.01, 0.28, 1.0}
\definecolor{bluet}{rgb}{0.1, 0.1, 0.8}

\newmdenv[innerlinewidth=1pt, roundcorner=0.5pt,linecolor=mycolor,innerleftmargin=0.5pt,
innerrightmargin=0.5pt,innertopmargin=0.4pt,innerbottommargin=0.4pt]{mybox}

\DeclareMathAlphabet{\pazocal}{OMS}{zplm}{m}{n}

\usepackage{bm}

\newcommand{\ba}{\begin{array}}
\newcommand{\ea}{\end{array}}
\newcommand{\be}{\begin{equation}}
\newcommand{\ee}{\end{equation}}
\newcommand{\ben}{\begin{equation*}}
\newcommand{\een}{\end{equation*}}
\newcommand{\bd}{\begin{displaymath}}
\newcommand{\ed}{\end{displaymath}}
\newcommand{\bi}{\begin{itemize}}
\newcommand{\ei}{\end{itemize}}
\newcommand{\bn}{\begin{enumerate}}
\newcommand{\en}{\end{enumerate}}
\newcommand{\p}{\partial}
\newcommand{\f}{\frac}

\newcommand{\mb}{\mathbf}
\newcommand{\fou}{\widetilde}
\newcommand{\sign}{\text{sign}}
\newtheorem{remark}{Remark}
\newcommand{\sw}[1]{{\color{black} \textit{#1}}}
\newcommand{\kd}[1]{{\color{black} \textit{#1}}}

\newcommand{\rev}[1]{{\color{black}#1}}
\newcommand{\revv}[1]{{\color{black}#1}}
\newcommand{\Sx}{S_x}
\newcommand{\Sy}{S_y}
\newcommand{\Fh}{\mathcal{H}}

\begin{document}

\title{On the stability analysis of perfectly matched layer for the elastic wave equation in layered media}
\author{
  Kenneth Duru\thanks{Mathematical Sciences Institute, Australian National University, Australia and Department of Mathematical Sciences, University of Texas at El Paso, TX, USA
  ({kenneth.duru@anu.edu.au}).}
  \and Balaje Kalyanaraman\thanks{Department of Computing Science and Department of Mathematics and Mathematical Statistics, Ume{\aa} University, Ume\aa,  Sweden
  ({balaje.kalyanaraman@umu.se}).}
  \and Siyang Wang\thanks{Department of Mathematics and Mathematical Statistics, Ume{\aa} University, Ume\aa,  Sweden
  ({siyang.wang@umu.se}). Corresponding author.}}
\maketitle

\begin{abstract}
In this paper, we present the stability analysis of the perfectly matched layer (PML) in two-space dimensional layered elastic media. Using normal mode analysis we prove that all interface wave modes present at a planar interface of  bi-material elastic solids  are dissipated by the PML. Our analysis builds upon the ideas presented in [SIAM Journal on Numerical Analysis 52 (2014) 2883–2904] and extends the stability results of  boundary waves (such as Rayleigh waves) on a half-plane elastic solid to   interface wave modes (such as Stoneley waves) transmitted into the PML at a planar interface separating two half-plane elastic solids. Numerical experiments in two-layer and multi-layer elastic solids corroborate the theoretical analysis, and generalise the results to complex elastic media. Numerical examples using the Marmousi model demonstrates the utility of the PML and our numerical method for seismological applications.
\end{abstract}


\noindent \textbf{ AMS}: 65M06, 65M12

\section{Introduction}
\label{s:1}

Wave motion is prevalent in many applications and has a great impact on our daily lives. 
Examples include the use of seismic waves \cite{Aki_and_Richards_1981,10.1785/BSSA0530010217B} to image natural resources in the Earth's subsurface, to detect cracks and faults in structures, to monitor underground explosions, and to investigate strong ground motions from earthquakes. 

Most wave propagation problems are formulated in large or infinite domains. However, because of limited computational resources, numerical simulations must be restricted to smaller computational domains by introducing artificial boundaries. Therefore, reliable and efficient domain truncation techniques that significantly minimise artificial reflections are important for the development of effective numerical wave solvers.
A straightforward approach to construct a domain truncation procedure is to surround the computational domain with an absorbing layer of finite thickness such that outgoing waves are absorbed. For this approach to be effective, all outgoing waves entering the layer should decay without  reflections, regardless of the frequency and angle of incidence. An absorbing layer with this desirable property is called a perfectly matched layer (PML) \cite{BERENGER1994185,Chew_and_Weedon,Kuzuoglu_and_Mittra,Duru2014K,Duru2012JSC,BECACHE2003399,APPELO2006642,Baffet2019,BecaheKachanovska2021}. 

The PML was first derived for electromagnetic waves in the pioneering work \cite{BERENGER1994185,Chew_and_Weedon} but has since then been extended  to other applications, for example acoustic and elastic waves \cite{Duru2014K,Duru2012JSC,BECACHE2003399,APPELO2006642,Baffet2019,BecaheKachanovska2021}. 
The PML has gained popularity because of its effective absorption properties, versatility, simplicity, ease of derivation and implementation using standard numerical methods. A stable PML model, when effectively implemented in a numerical solver, can yield a domain truncation scheme that ensures the convergence of the numerical solution to the solution of the unbounded problem \cite{Baffet2019,DURU2019898,ElasticDG_PML2019}.
However, the PML is also notorious for supporting fatal instabilities which can destroy the accuracy of numerical solutions. These undesirable exponentially growing modes can be present in both the PML model at the continuous level or numerical methods at the discrete level.  

The stability analysis of the PML has attracted substantial attention in the literature, see for example \cite{Duru2014K,Duru2012JSC,BECACHE2003399,APPELO2006642,Baffet2019,BecaheKachanovska2021}, and \cite{KDuru_and_GKreiss_Review_2022} for a recent review. 
%
%
For hyperbolic PDEs, mode analysis for PML initial value problems (IVP) with constant damping and constant material properties  yields a necessary geometric stability condition \cite{BECACHE2003399}. When this condition is violated, exponentially growing modes in time are present, rendering the PML model useless. In certain cases, for example the acoustic wave equation with constant coefficients, analytical solutions can be derived \cite{BecaheKachanovska2021,DIAZ20063820}. In addition, energy estimates for the PML have recently been derived in physical space 
\cite{Baffet2019} and Laplace space \cite{DURU2019898,ElasticDG_PML2019}, which can be useful for deriving stable numerical methods. 
However, in general, even if the PML IVP does not support growing modes, there can still be stability issues when  boundaries and material interfaces are introduced. For the extension of mode stability analysis to boundary and guided waves in homogeneous media, see \cite{Duru2014K,DURU2015372,Duru2012JSC, DURU2014445}. The stability analysis of the PML in discontinuous acoustic  media  was presented in  \cite{DURU2014757}. To the best of our knowledge, the stability analysis of the PML for more general wave media such as the discontinuous or layered elastic solids has not been reported in literature. 

In geophysical and seismological applications, the wave media can be composed of layers of rocks, soft and hard sediments, bedrock layers,   water and possibly oil. In layered elastic media, the presence of interface wave modes such as Stoneley waves \cite{doi:10.1098/rspa.1924.0079,10.1785/BSSA0530010217B}, makes the stability analysis of the PML more challenging. 
Numerical experiments have also reported PML instabilities and poor performance for problems with  material boundaries entering into the layer and problems with strong evanescent waves \cite{APPELO20094200}.  These existing results have motivated this study to investigate where the  inadequacies of the PML arise.  

The main objective of this study is to analyse the stability of interface wave modes for the PML in discontinuous elastic solids. Using normal mode analysis, we prove that  if the PML IVP has no temporally growing modes, then all interface wave modes present at a planar interface of  bi-material elastic solids  are dissipated by the PML. The analysis closely follows the steps taken in \cite{Duru2014K} for  boundary waves modes, but here we apply the techniques to investigate the stability of interface wave modes in the PML.  Numerical experiments in two-layered isotropic and anisotropic elastic solids, and a multi-layered isotropic elastic solid corroborate the theoretical analysis. Furthermore, we present numerical examples using the Marmousi model \cite{marmousi2}, extending the results to complex elastic media and demonstrating the utility of the PML and our numerical method for seismological applications.

The remainder of the paper proceeds as follows. In section~\ref{sec:intro} we present the elastic wave equation in discontinuous media, define interface conditions and discuss energy stability for the model problem. 
In section \ref{sec:mode_analysis}, we introduce the mode analysis for body and interface wave modes, and formulate the determinant condition that is necessary for stability. The PML model is derived in section~\ref{sec:pml_model}. In section \ref{sec:stability_pml}, we present the stability analysis of the PML in a piecewise constant elastic medium and formulate the main results. Numerical examples are given in section \ref{sec:num_exp}, corroborating the theoretical analysis. In section~\ref{sec:conclusions}, we draw conclusions.

\section{The elastic wave equation in discontinuous media}
\label{sec:intro}
Consider the  2D elastic wave equation in Cartesian coordinates $(x,y) \in \Omega \subseteq \mathbb{R}^2$,
{
\begin{align}
    \rho\frac{\partial^2 {\mathbf{u}}}{\partial t^2} & = 
    \frac{\partial}{\partial x}{\mathbf{T}}_x + 
        \frac{\partial}{\partial y}{\mathbf{T}}_y, \quad {\mathbf{T}}_x  = A \frac{\partial {\mathbf{u}}}{\partial x} + C \frac{\partial {\mathbf{u}}}{\partial y},
\quad
{\mathbf{T}}_y = C^T \frac{\partial {\mathbf{u}}}{\partial x} + B \frac{\partial {\mathbf{u}}}{\partial y},
        \label{pde1}
\end{align}}%
where $\mathbf{u} = [u_1, u_2] \in \mathbb{R}^2$ is the displacement vector,  ${\mathbf{T}}_x, {\mathbf{T}}_y \in \mathbb{R}^2$ are the   stress vectors and $t\ge 0$ denotes time. 
We set the smooth initial conditions
{
\begin{align}\label{initial_con}
\mb{u}|_{t =0} = \mb{f}\in H^1{(\Omega)}, \quad \f{\p \mb{u}}{\p t}|_{t =0} = \mb{g}\in L^2{(\Omega)}.
\end{align}
}
The medium parameters are described by the density $\rho >0$ and the coefficient matrices $A$, $B$, $C$ of elastic constants.  In 2D orthotropic elastic media, the elastic coefficients are described by four independent parameters $c_{11}$, $c_{22}$, $c_{33}, c_{12}$ and the coefficient matrices are given by 
\begin{equation}\label{mp}
A = \begin{bmatrix}
c_{11} & 0 \\
0 & c_{33}
\end{bmatrix}, \quad B = \begin{bmatrix}
c_{33} & 0 \\
0 & c_{22}
\end{bmatrix}, \quad C = \begin{bmatrix}
0 & c_{12} \\
c_{33} & 0
\end{bmatrix},
\end{equation}
Here, the material coefficients $c_{11}$, $c_{22}$, $c_{33}$ are always positive, but $c_{12}$ may be negative for certain materials. In general, for stability, we require
\begin{align}\label{elastic_contants_ellipticity}
    c_{11} > 0, \quad  c_{22} >0, \quad  c_{33} >0, \quad c_{11}c_{22}-c_{12}^2>0.
\end{align}
For planar waves propagating along the $x$-direction and $y-$direction, the $p$-wave speeds and $s$-wave speeds are given by 
\begin{align}\label{elastic_wave_speeds_an}
     c_{px}:= \sqrt{\frac{c_{11}}{\rho}}, \quad  c_{sx}:= \sqrt{\frac{c_{33}}{\rho}}, \quad c_{py}:= \sqrt{\frac{c_{22}}{\rho}}, \quad  c_{sy}:= \sqrt{\frac{c_{33}}{\rho}}.
\end{align}

In the case of isotropic media, the material properties can be described by only two Lam\'{e} parameters, $\mu >0$ and $\lambda$, such that $c_{11}=c_{22}=2\mu+\lambda$, $c_{33}=\mu>0$, $c_{12}=\lambda > - \mu$, yielding 
\begin{equation}\label{mp_isotropic}
A = \begin{bmatrix}
2\mu+\lambda & 0 \\
0 & \mu
\end{bmatrix}, \quad B = \begin{bmatrix}
\mu & 0 \\
0 & 2\mu+\lambda
\end{bmatrix}, \quad C = \begin{bmatrix}
0 & \lambda \\
\mu & 0
\end{bmatrix},
\end{equation}
with the wave speeds 
\begin{align}\label{elastic_wave_speeds_iso}
     c_{p}:= \sqrt{\frac{2\mu+\lambda}{\rho}}, \quad  c_{s}:= \sqrt{\frac{\mu}{\rho}}.
\end{align}
In isotropic media, a wave mode  propagates with the same wave speed in all directions.

In general, the elastic medium modeling the solid Earth can have a layered structure with piecewise smooth material property. At material discontinuities, we define the outward unit normal vector $\mathbf{n}= (n_x, n_y)^T\in \mathbb{R}^2$  on the interface, and the traction vector
$$
{\mathbf{T}} = n_x {\mathbf{T}}_x + n_y{\mathbf{T}}_y.
$$ 
Then, we enforce the interface conditions such that the displacement and traction are continuous,
\begin{align}
\label{int_con}
    \lJump{{\mathbf{T}} \rJump} =0, \quad \lJump{{\mathbf{u}} \rJump} =0.
\end{align}
Note that if the discontinuity lies on the 
\sw{$x$-axis ($y$-axis)}, we have $(n_x, n_y) = (0, 1)$  and ${\mathbf{T}} = {\mathbf{T}}_y$ ($(n_x, n_y) = (1, 0)$ and ${\mathbf{T}}  ={\mathbf{T}}_x$).

We introduce the strain-energy matrix
\begin{align}\label{eq:energy_matrix}
P = 
\begin{bmatrix}
A & C\\
C^T & B
\end{bmatrix}.
\end{align}
The symmetric strain-energy matrix $P$ is positive semi-definite \cite{Lai1993}.
We define the mechanical energy in the medium $\Omega$ by
\begin{align}\label{eq:energy_continuous}
E(t) = \frac{1}{2}\int_{\Omega}\left(
\rho\left(\f{\p\mb{u}}{\p t}\right)^T\left(\f{\p\mb{u}}{\p t}\right) +
\begin{bmatrix}
\f{\p\mb{u}}{\p x}\\
\f{\p\mb{u}}{\p y}
\end{bmatrix}^T 
P
\begin{bmatrix}
\f{\p\mb{u}}{\p x}\\
\f{\p\mb{u}}{\p y}
\end{bmatrix}\right)dxdy.
\end{align}
The following theorem states a stability result for the coupled problem, \eqref{pde1}--\eqref{initial_con} and \eqref{int_con}. 
\begin{theorem}\label{Thm_energy}
\sw{Consider the elastic wave equation \eqref{pde1} with piecewise smooth material parameters defined on the whole plane $\Omega = \mathbb{R}^2$, subject to the initial condition \eqref{initial_con}, and the interface conditions \eqref{int_con} at material discontinuities. Assuming the decay condition that for any fixed $t>0$, the solution $|\mathbf{u}(x,y,t)| \to 0$ for $|(x,y)| \to +\infty$, the total mechanical energy $E(t)\geq 0$ given in \eqref{eq:energy_continuous} is conserved, i.e., $E(t)=E(0)$ for all $t\ge 0$.}
\end{theorem}
The proof of Theorem \ref{Thm_energy} is standard and can be adapted from \cite{Duru2014V}. 
We say that the problem, \eqref{pde1}--\eqref{initial_con} and \eqref{int_con}, is energy-stable if $E(t)\le E(0)$ for all $t\ge 0$. \sw{Note that  the interface conditions \eqref{int_con} specifies the minimal number of coupling conditions,  thus  the energy-stability $E(t)\le E(0)$    also translates to the well-posedness of the model.}

\section{Layered elastic media and mode analysis}\label{sec:mode_analysis}
As we will see later in this paper, when the PML is introduced  the equations will become asymmetric and the energy method may not be applicable to establishing the stability of the model. Next we will consider layered elastic media by assuming variations of media parameters along the $y$-axis only and introduce the mode analysis which will be useful in analysing the stability of the PML in layered elastic media. A similar approach was taken in \cite{Duru2014K} to analyse the stability of  boundary waves modes in the PML, but here we apply the techniques to investigate the stability of interface elastic wave modes, such as Stoneley waves, in the PML. 
\subsection{Layered elastic media}
In the coming analysis we will focus on stratified elastic media with discontinuities along the \revv{$x$-axis}. 
In particular, we will consider the 2D elastic wave equation \eqref{pde1} in the two  half-planes $\Omega = \Omega_1 \cup \Omega_2$, with $\Omega_1 = (-\infty, \infty)\times(0, \infty)$ and $\Omega_2 = (-\infty, \infty)\times(-\infty, 0)$, and a discontinuous planar interface at $y =0$.
In each half-plane $\Omega_i$, for $i = 1,2$, we denote the displacement field \sw{by} $\mb{u}_i$ and the piecewise constant medium parameters \sw{by} $\rho_i >0$  $A_i$, $B_i$, $C_i$.
At the material interface $y=0$, 
\sw{interface conditions \eqref{int_con} have} a more explicit form
\begin{equation}\label{int_con_2}
B_1\f{\p\mb{u}_1}{\p y}+C_1^T\f{\p\mb{u}_1}{\p x}=B_2\f{\p\mb{u}_2}{\p y}+C_2^T\f{\p\mb{u}_2}{\p x}, \quad \mb{u}_1=\mb{u}_2.
\end{equation}
\subsection{Mode analysis}
Theorem \ref{Thm_energy} proves energy stability of the elastic wave equation \eqref{pde1} in piece-wise smooth elastic media $\Omega_i$, for $i = 1, 2$, subject to the interface condition \eqref{int_con} or \eqref{int_con_2}. However, the theorem does not provide information about the wave modes that may exist in the medium. In this section, we use mode analysis to gain insights on the existence of possible wave modes. More precisely, we start by considering a constant-coefficient problem for the existence of body waves. After that, we analyse interface waves in media with piecewise constant material property and formulate a stability result in the framework of normal mode analysis. This mode analysis framework will be useful when proving the stability of the PML at the presence of interface wave modes. 

\subsubsection{Plane waves and dispersion relations}\label{sec:dispersion_relation}
To study the existence of body wave modes, we consider the problem \eqref{pde1} in the whole real plane  $(x,y) \in \mathbb{R}^2$ with constant medium parameters, 
$$
\rho_i = \rho, \quad A_i = A, \quad B_{i} = B, \quad C_i = C,
$$
for $\Omega_i$, $i = 1, 2$. 
\kd{In this case, we do not consider $y=0$ as a material interface.}

Consider the wave-like solution
\begin{align}\label{eq:plane_wave}
\mathbf{u}\left(x, y, t\right) = \mathbf{u}_0 e^{st+\bm{i}\left( k_x x  + k_y y \right)}, \quad \mathbf{u}_0 \in \mathbb{R}^2, \quad k_x, k_y, x, y \in \mathbb{R}, \quad t\ge 0, \quad \bm{i} = \sqrt{-1}.
\end{align}
In \eqref{eq:plane_wave}, $\mathbf{k}=\left(k_x, k_y\right) \in \mathbb{R}^2$ is the wave vector, and $ \mathbf{u}_0\in \mathbb{R}^2$ is a vector of constant amplitude
called the polarization vector. 
By inserting \eqref{eq:plane_wave} into \eqref{pde1}, we have the eigenvalue problem
\begin{align}\label{body_wave_speeds_eig}
-s^2\mathbf{u}_0 = \mathcal{P}(\mathbf{k})\mathbf{u}_0, \quad  \mathcal{P}(\mathbf{k}) = \frac{k_x^2 A + k_y^2 B + k_xk_y(C + C^T)}{\rho}.
\end{align}
The polarisation vector $\mathbf{u}_0 \in \mathbb{R}^2$ is an eigenvector of the matrix $\mathcal{P}(\mathbf{k})$ and $-s^2$ is the corresponding eigenvalue.
For problems that are energy conserving,  the matrix $\mathcal{P}(\mathbf{k})$ is symmetric positive definite for all $\mathbf{k} \in \mathbb{R}^2$. Thus, the eigenvectors $\mathbf{u}_0$ of $\mathcal{P}(\mathbf{k})$ are orthogonal and the eigenvalues are real and positive, $-s^2 >0$.   
The wave-mode \eqref{eq:plane_wave}, defined by the eigenpair $s, \mathbf{u}_0$, is a solution of the elastic wave equation \eqref{pde1} in the whole plane  $(x,y) \in \mathbb{R}^2$ if $s$ and $\mathbf{k}$ satisfy the dispersion relation 
\begin{align}\label{eq:dispersion_relation_f}
&F\left(s,  \mathbf{k}\right):= \det\left({s^2} I +  \mathcal{P}(\mathbf{k})\right) = 0.
\end{align}
Evaluating the determinant and simplifying further, we obtain 
{
\footnotesize
\begin{align}\label{eq:dispersion_orthotropic}
F\left(s,  \mathbf{k}\right) = s^4 + \frac{\left(c_{11} + c_{33}\right)k_x^2+\left(c_{22} + c_{33}\right)k_y^2}{\rho} s^2
+
\frac{c_{11}c_{33}k_x^4 +  c_{22}c_{33}k_y^4 + \left(c_{11}c_{22} + c_{33}^2 - \left(c_{33}+c_{12}\right)^2\right)k_x^2k_y^2}{\rho^2} = 0.
\end{align}
}
In an isotropic medium, with $c_{11}=c_{22}=2\mu+\lambda$, $c_{33}=\mu>0$, $c_{12}=\lambda > - \mu$, the dispersion relation simplifies to
\begin{align}\label{eq:dispersion_isotropic}
F\left(s,  \mathbf{k}\right)=\left({s^2} + c_p^2|\mathbf{k}|^2\right)\left({s^2}+ c_s^2|\mathbf{k}|^2\right)=0, \quad 
     c_p= \sqrt{\frac{2\mu + \lambda}{\rho}}, \quad  c_s= \sqrt{\frac{\mu}{\rho}}, \quad |\mathbf{k}| = \sqrt{k_x^2 + k_y^2}.
\end{align}
Then, the eigenvalues are given by
\begin{equation}\label{eq:eigenvalues_isotropic}
\begin{split}
-s_1^2 = c_p^2|\mathbf{k}|^2, \quad -s_2^2= c_s^2|\mathbf{k}|^2,
\end{split}
\end{equation}
which correspond to the P-wave and S-wave propagating in the medium.
In linear orthotropic  elastic media, the eigenvalues  $-s^2$ also have  closed form expressions
{
\footnotesize
\begin{equation}\label{eq:eigenvalues_orthotropic}
\begin{split}
-s_1^2 &= \frac{1}{2{\rho}}\left({(c_{11}+c_{33})k_x^2 + (c_{22}+c_{33})k_y^2}\right)\\
&+
\frac{1}{2{\rho}}\sqrt{\left({(c_{11}+c_{33})k_x^2 + (c_{22}+c_{33})k_y^2}\right)^2 -4\left(\left(c_{11}c_{33}k_x^4
+c_{22}c_{33}k_y^4\right)+ \left(c_{11}c_{22} +c_{33}^2-\left(c_{12} +c_{33}\right)^2k_x^2k_y^2\right)\right)},\\
-s_2^2 &= \frac{1}{2{\rho}}\left({(c_{11}+c_{33})k_x^2 + (c_{22}+c_{33})k_y^2}\right)\\
&-
\frac{1}{2{\rho}}\sqrt{\left({(c_{11}+c_{33})k_x^2 + (c_{22}+c_{33})k_y^2}\right)^2 -4\left(\left(c_{11}c_{33}k_x^4
+c_{22}c_{33}k_y^4\right)+ \left(c_{11}c_{22} +c_{33}^2-\left(c_{12} +c_{33}\right)^2k_x^2k_y^2\right)\right)}.
\end{split}
\end{equation}
}
Using the stability conditions \eqref{elastic_contants_ellipticity}, it is easy to check that  the two eigenvalues are strictly positive, that $-s_j^2 > 0$ for $j = 1,2$.
These  two eigenvalues again indicate two body-wave modes, corresponding to the quasi-P waves and the quasi-S waves. 

\begin{remark}\label{remark:imaginaryroots}
The indeterminate $s\in \mathbb{C}$  that solves the dispersion relation \eqref{eq:dispersion_relation_f} is related to the temporal frequency.
Since Theorem \ref{Thm_energy} holds for all stable  medium parameters, the whole plane problem \eqref{pde1} conserves energy. 
Thus, the real part of the roots $s$ must be zero, that is $s\in \mathbb{C}$ with $\Re{s} =0$.
Otherwise, if the roots $s$ have non-zero real parts then the energy will grow or decay, contradicting Theorem \ref{Thm_energy}.
\end{remark}

We write $s =  \bm{i}\omega$, where $\omega \in \mathbb{R}$ is called the temporal frequency, and introduce
\begin{align}\label{KVSV}
\begin{split}
&\mathbf{K} = \left(\frac{k_x}{|\mathbf{k}|}, \frac{k_y}{|\mathbf{k}|}\right), \quad \text{normalised propagation direction},\\
&\mathbf{V}_p = \left(\frac{\omega}{k_x}, \frac{\omega}{k_y}\right), \quad \text{phase velocity},\\
&\mathbf{\mathcal{S}} = \left(\frac{k_x}{\omega}, \frac{k_y}{\omega}\right), \quad \text{slowness vector},\\
&\mathbf{V}_g = \left(\frac{\partial \omega}{\partial k_x}, \frac{\partial \omega}{\partial k_y}\right), \quad \text{group velocity}.
\end{split}
\end{align}
For the Cauchy problem in a constant coefficient medium, the dispersion relation $F\left(i\omega,  \mathbf{k}\right)=0$   and the quantities $\mathbf{K}$, $\mathbf{V}_p $, $\mathbf{\mathcal{S}}$, $\mathbf{V}_g$, defined above give detailed description of the wave propagation properties in the medium. In addition, they determine a stability property for the corresponding PML model, which is discussed in section 5.1. 
In Figure \ref{fig:dispersion_relation}, we plot the dispersion relations of two different elastic solids, showing the slowness diagrams.
\begin{figure}[ht]
 \centering
%
%
%
\begin{subfigure}{0.48\textwidth}
    \includegraphics[width=\textwidth]{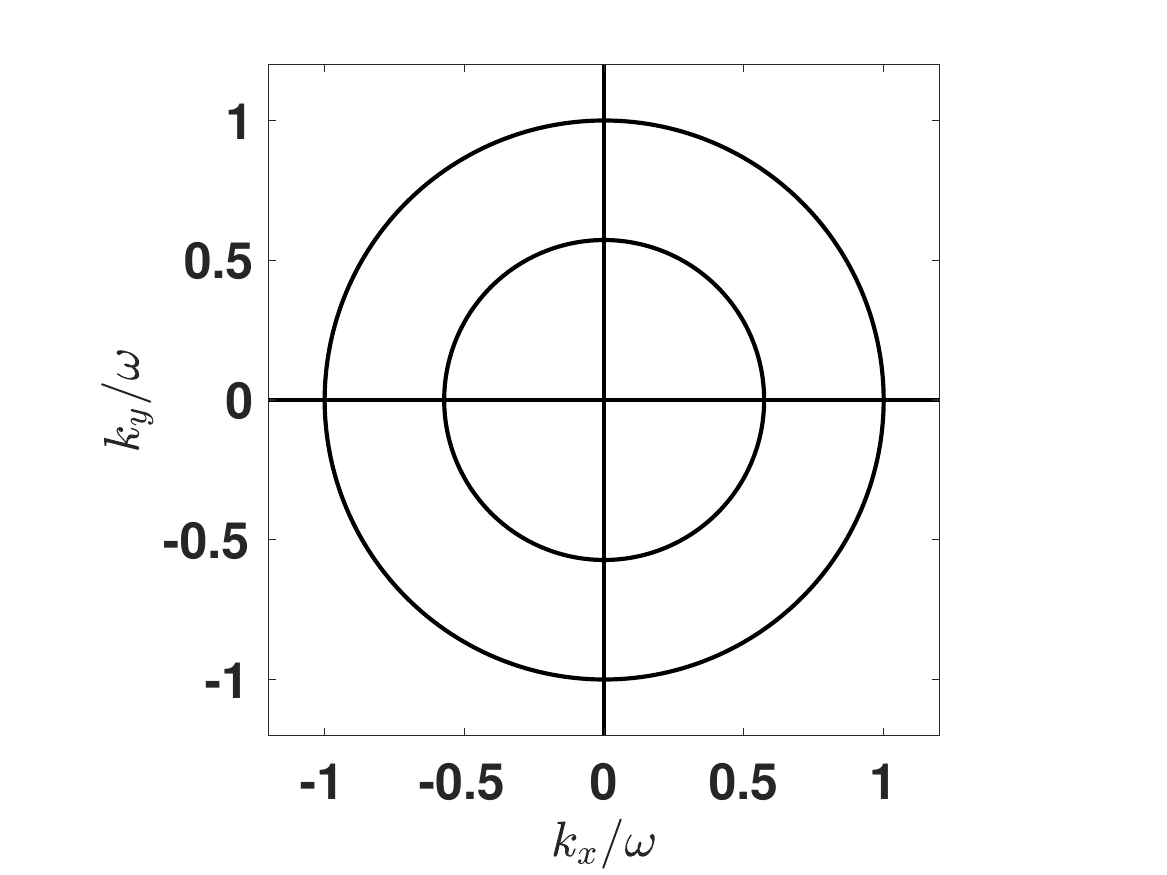}
\end{subfigure}
\begin{subfigure}{0.48\textwidth}
    \includegraphics[width=\textwidth]{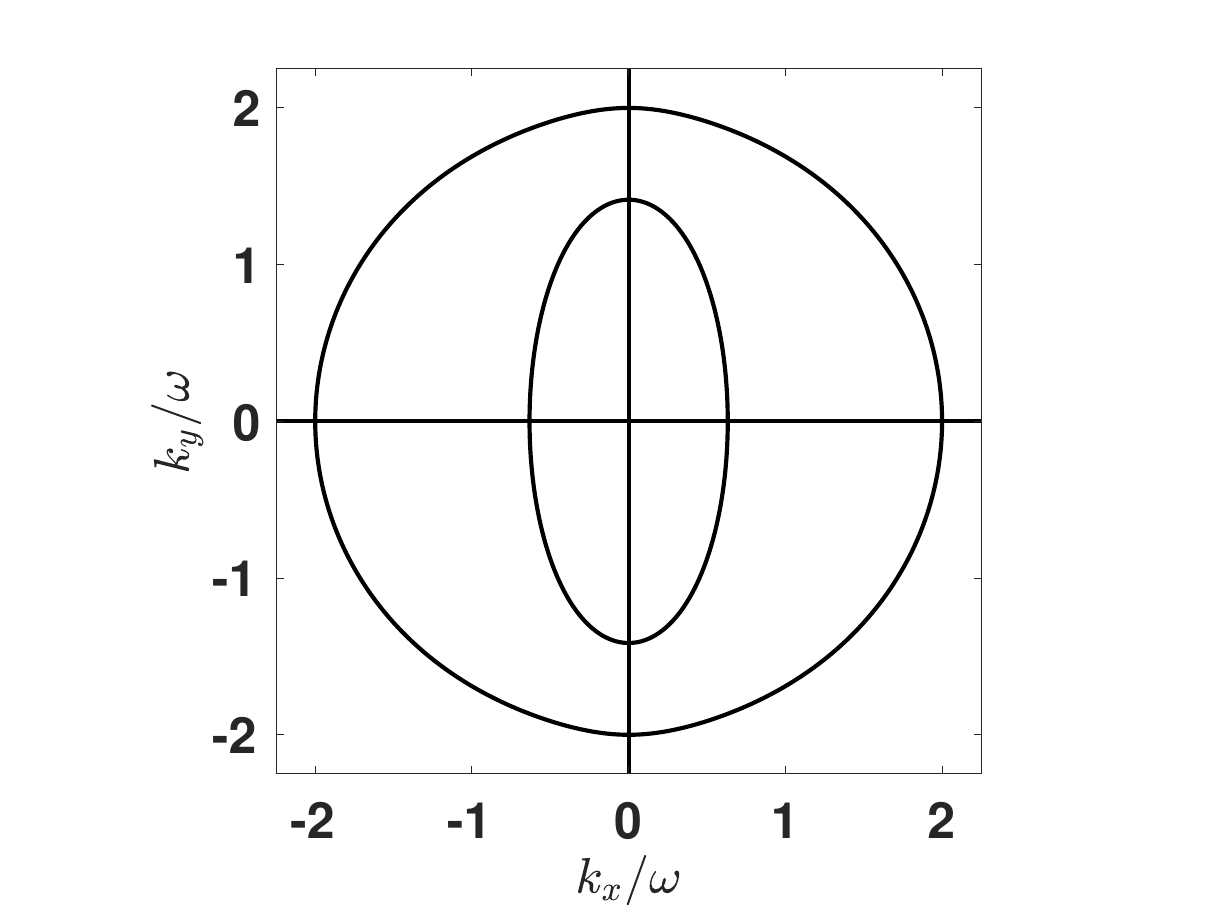}
\end{subfigure}
\caption{Slowness diagrams showing the solutions of the dispersion relations, $F\left(i,  \mathbf{S}\right) = 0$,  for an isotropic and anisotropic  elastic wave media in two space dimensions. For each figure, the two  curves correspond to S-wave and P-wave modes, respectively.}
 \label{fig:dispersion_relation}
\end{figure}

When  boundaries and interfaces are present, additional boundary and interface wave modes, such as Rayleigh \cite{https://doi.org/10.1112/plms/s1-17.1.4}  and Stoneley waves \cite{doi:10.1098/rspa.1924.0079,10.1785/BSSA0530010217B}, are introduced. 
In the following, we consider the problem in two half-planes coupled together at a planar interface and  formulate an alternative procedure to characterise the stability property of interface wave modes. 

\subsubsection{Normal mode analysis and the determinant condition}\label{sec:normal_modes_analysis_interface}
Here, we present the normal mode analysis for  interface wave modes in discontinuous media, which tightly connects to the analysis of the PML in the next section. We note that a similar approach was taken in \cite{Duru2014K} to investigate the stability of boundary waves in a half plane elastic solid.
However, we will include the details here to ensure a self-contained narrative.
To begin, we consider piecewise constant media parameters
$$
\rho_i >0, \quad A_i, \quad B_{i}, \quad C_i,
$$
for $\Omega_i$, $i = 1, 2$, where $A_i, B_{i}, C_i$ are given in \eqref{mp}. The material parameters are constant in each half-plane, but are discontinuous  at the interface  $y =0$,  where the equations \eqref{pde1} are coupled by the interface condition  \eqref{int_con}. We look for  wave solutions in the form
\begin{align}\label{eq:simple_wave}
\mathbf{u}_i\left(x, y, t\right) = \boldsymbol{\phi}_i(y) e^{st+ \bm{i} k_x x   }, \quad \|\boldsymbol{\phi}_i\| < \infty, \quad k_x\in \mathbb{R}, \quad (x, y) \in \Omega_i, \quad t\ge 0.
\end{align}



\sw{Again the wave modes \eqref{eq:simple_wave}, defined by the eigenpairs $s, \boldsymbol{\phi}_i(y)$, are the solutions of the half-plane problems. The eigenpairs $s, \boldsymbol{\phi}_i(y)$ will be determined by solving a nonlinear eigenvalue problem.}
The variable $s$ is related to the well-posedness and stability properties of the model, which are defined as follows.
\begin{definition}\label{def:well-posedness_and_stability}
    \sw{
    An initial boundary value problem with the solution \eqref{eq:simple_wave} is well-posed if the variables $\revv{s(k_x) \in \mathbb{C}}$ and $k_x \in \mathbb{R}$  satisfy
    \[
\lim_{|k_x|\rightarrow\infty} \frac{\Re{s}}{|k_x|} \leq 0.
    \]
    Otherwise, the problem is ill-posed. In addition, we call the solution \eqref{eq:simple_wave} stable if $\Re{s} \leq 0$ for all $k_x \in \mathbb{R}$. }
\end{definition}
\sw{By Definition \ref{def:well-posedness_and_stability}, a well-posed problem can potentially allow a bounded exponentially growing solution of the form \eqref{eq:simple_wave}. However, a stable problem excludes exponentially growing solution of the form \eqref{eq:simple_wave}, and must require that $\Re{s} \leq 0$ for all $k_x \in \mathbb{R}$. }

\sw{Note in particular that because the elastic wave equations \eqref{pde1} has no lower order terms we can make a more precise statement about the well-posedness of the transmission problem, that is the elastic wave equation  \eqref{pde1} with the interface condition  \eqref{int_con} or \eqref{int_con_2}.}
\begin{lemma}\label{def:well-posedness}
The elastic wave equations \eqref{pde1} with piecewise constant media parameters \eqref{mp}  and the interface condition  \eqref{int_con} or \eqref{int_con_2} are not \sw{well-posed} in any sense if there are nontrivial solutions of the form \eqref{eq:simple_wave} with $\Re{s} >0$. 
\end{lemma}
\begin{proof}
   \sw{The proof  can be adapted from the proof of Lemma 3.1 in \cite{Kreiss2012}. 
   If there is a solution of the form \eqref{eq:simple_wave} with $\Re{s} >0$, then 
   \begin{align*}
\mathbf{u}_{\gamma i}\left(x, y, t\right) = \boldsymbol{\phi}_i(\gamma y) e^{s \gamma t+ \bm{i} \gamma k_x x   }, \quad  \gamma >0,
\end{align*}
for any $\gamma >0$ is also a solution. Since $\Re{s} >0$ we can find solutions that grow arbitrarily fast exponentially.}
\end{proof}
If there are nontrivial solutions of the form \eqref{eq:simple_wave} with $\Re{s}>0$, we can always construct solutions that grow arbitrarily fast \sw{exponentially}, which is not supported by a \sw{ well-posed system. Apparently the loss of well-posedness translates to uncontrollable instability.}
Now, we reformulate Lemma \ref{def:well-posedness} as an algebraic condition, i.e. the so-called determinant condition in Laplace space \cite{Duru2014K, Gustafsson2013}.

For a complex number $z = a + \bm{i}b$, we define the branch of $\sqrt{z}$ by
$$
-\pi < \arg{(a + \bm{i}b)} \le \pi, \quad \arg{(\sqrt{a + \bm{i}b})} = \frac{1}{2} \arg{(a + \bm{i}b)}.
$$
We insert \eqref{eq:simple_wave} in the equation \eqref{pde1} and the interface condition  \eqref{int_con}, and obtain
\begin{align}
s^2 \rho_i{\boldsymbol{\phi}_i} = -k_x^2A_i \boldsymbol{\phi}_i
+B_i\frac{d^2\boldsymbol{\phi}_i}{dy^2}
+\bm{i}k_x \left(C_i + C^T_i\right) \frac{d\boldsymbol{\phi}_i}{dy},\quad i=1,2,\label{pde1_lap} \\
\boldsymbol{\phi}_1 =\boldsymbol{\phi}_2,\quad
B_1\f{d{\boldsymbol{\phi}_1}}{d y}+\bm{i}k_xC_1^T{{\boldsymbol{\phi}_1}}=B_2\f{d\boldsymbol{\phi}_2}{d y}+\bm{i}k_xC_2^T{\boldsymbol{\phi}_2},\quad  \sw{\text{at} y=0}.\label{int_con_Lap}
\end{align}
For $\boldsymbol{\phi}_i$, we seek the modal solution
\begin{equation}\label{eq:modal_solution}
\boldsymbol{\phi}_i = \boldsymbol{\Phi}_i e^{\kappa y}, \quad \boldsymbol{\Phi}_i \in \mathbb{C}^2,\quad i=1,2.
\end{equation}
Inserting the modal solution \eqref{eq:modal_solution} in  \eqref{pde1_lap}, we have the eigenvalue problem
\begin{align}\label{interface_wave_speeds_eig}
-s^2\boldsymbol{\Phi}_i = \mathcal{P}_i(k_x, \kappa)\boldsymbol{\Phi}_i, \quad  \mathcal{P}_i(k_x, \kappa) = \frac{k_x^2 A_i - \kappa^2 B_i - \bm{i}k_x\kappa(C_i + C_i^T)}{\rho_i},\quad i=1,2.
\end{align}
The solutions satisfy the condition
\begin{align}\label{eq:characteristics_f}
&F_i\left(s,  k_x, \kappa\right):= \det\left({s^2} I +  \mathcal{P}_i(k_x, \kappa)\right) = 0.
\end{align}

 For a fixed $k_x \in \mathbb{R}$ and $s$ with sufficiently large $\Re{s} > 0$,  the roots $\kappa_i$ come in pairs and have non-vanishing real parts, \cite{Duru2014K}, with
\sw{
\begin{align}\label{eq:kappa_roots}
&\kappa_{i1}^{\pm} = \pm \sqrt{\frac{\left(c_{11_i}c_{22_i}-(c_{12_i}^2 + 2c_{12_i}c_{33_i})\right)k_x^2 + (c_{11_i} + c_{33_i})^2 \rho_i s^2 - \gamma_i(s, k_x)}{2c_{22_i}c_{33_i}}},
\\
\nonumber
&\kappa_{i2}^{\pm} = \pm \sqrt{\frac{\left(c_{11_i}c_{22_i}-(c_{12_i}^2 + 2c_{12_i}c_{33_i})\right)k_x^2 + (c_{11_i} + c_{33})^2 \rho_i s^2 + \gamma_i(s, k_x)}{2c_{22_i}c_{33_i}}}
\end{align}
where
{
\scriptsize
$$
\gamma_i(s, k_x) =  \sqrt{\left(\left(c_{11_i}c_{22_i}-(c_{12_i}^2 + 2c_{12_i}c_{33_i})\right)k_x^2 + (c_{11_i} + c_{33_i})^2 \rho_i s^2\right)^2 - 4c_{22_i}c_{33_i}\left((c_{11_i} + c_{33_i})^2 \rho_i s^2k_x^2 + c_{11_i}c_{33_i}k_x^4 + \rho_i^2 s^4\right)}.
$$
}
}
The concept of homogeneity of the roots will be important to analyse stability property. Below, we give its formal definition. 
\begin{definition}\label{def_homogeneous}
Let $\bm{f}(\bm{v})$ be a function with the vector argument $\bm{v}$. If $\bm{f}(\alpha \bm{v}) = \alpha^n\bm{f}(\bm{v})$ for all nonzero scalar $\alpha \ne 0$ and some $n \in \mathbb{Z}$, then $\bm{f}(\bm{v})$ is homogeneous of degree $n$.
\end{definition}
\sw{
Note that the roots $\kappa_{ij}^{\pm}(s, k_x)$ are homogeneous of degree one in $(s, k_x)$.
} The following lemma states another important property of the roots.


\begin{lemma}
The real parts of the roots  $\kappa_{ij}^{\pm},\ i,j=1,2$ in \eqref{eq:kappa_roots} do not change sign for all $s$ with $\Re{s} >0$.
\end{lemma}

\begin{proof}
\sw{We prove by contradiction and assume that the real part of a root $\kappa_{ij}$ changes sign for some $s$ with $\Re{s} >0$. Since the root varies continuously with $s$, then for some $s$ with $\Re{s} >0$ the root is purely imaginary, $\kappa_{ij}  = \bm{i}k_y$. In this case, the dispersion relation \eqref{eq:characteristics_f} is the same as  \eqref{eq:dispersion_relation_f} for the Cauchy problem, which admits only \revv{purely} imaginary $s$. }
\end{proof}

\sw{Thus, for each $i=1,2$ there are exactly two roots with positive real parts, $\kappa_{ij}^{+}$ and  exactly two roots with negative real parts, $\kappa_{ij}^{-}$, for all $j = 1, 2$.} We use the notation \eqref{eq:kappa_roots} to denote the roots with the stated sign convention  for all $s$ with $\Re{s} >0$.
That is, the superscript ($+$) denotes the root with positive real part and the superscript ($-$) denotes the root with negative real part.  
Because of the condition $\|\boldsymbol{\phi}_i\|<\infty$, the general solution of \eqref{pde1_lap} takes the form
\begin{equation}\label{gen_sol_ped1_lap}
\boldsymbol{\phi}_1(y)=\delta_{11} e^{\kappa_{11}^- y}\Phi_{11}+\delta_{12} e^{\kappa_{12}^- y}\Phi_{12}, \quad \boldsymbol{\phi}_2(y) =\delta_{21} e^{\kappa_{21}^+ y}\Phi_{21}+\delta_{22} e^{\kappa_{22}^+ y}\Phi_{22},
\end{equation}
where $\Phi_{ij},\ i,j=1,2$ are the corresponding eigenvectors. As an example, in isotropic linear elastic media, the analytical expressions of the roots and eigenvectors are 
\begin{align*}
\kappa_{i1}^{\pm}=\pm \sqrt{ k_x^2+\f{ s^2}{c_{si}^2}}, \quad \kappa_{i2}^{\pm}=\pm \sqrt{k_x^2 + \f{s^2}{c_{pi}^2}},\quad i=1,2,
\end{align*}
 and 

\begin{align}\label{eq:eigenfunctions}
\Phi_{11}=\begin{bmatrix}
\f{\bm{i}}{k_x}\kappa_{11}^-\\
1
\end{bmatrix},\quad
\Phi_{12}=\begin{bmatrix}
\f{-\bm{i} k_x}{\kappa_{12}^-}\\
1
\end{bmatrix},\quad \Phi_{21}=\begin{bmatrix}
\f{\bm{i}}{ k_x}\kappa_{21}^+\\
1
\end{bmatrix},\quad \Phi_{22}=\begin{bmatrix}
\f{-\bm{i} k_x}{\kappa_{22}^+}\\
1
\end{bmatrix}.
\end{align}
For orthotropic elastic media, the roots can also be expressed in  closed form \eqref{eq:kappa_roots}, but the expressions are much more complicated. We refer the reader to \cite{Duru2014K} for more details. \sw{Because the roots are homogeneous of degree one in $(s, k_x)$, the eigenfunctions $\Phi_{ij}$ are homogeneous of degree zero in $(s, k_x)$.}



The coefficients $\boldsymbol{\delta}=[\delta_{11},\delta_{12},\delta_{21},\delta_{22}]^T$ are determined by inserting \eqref{gen_sol_ped1_lap} into the interface conditions \eqref{int_con_Lap}, yielding the following equation
\begin{equation}\label{Fs_pml}
\mathcal{C}(s,  k_x)\boldsymbol{\delta}=\mb{0},
\end{equation}
where the $4\times 4$ interface matrix $\mathcal{C}$ takes the form
{
\footnotesize
\begin{equation}\label{F_0}
\mathcal{C}(s,  k_x)=\begin{bmatrix}
\Phi_{11} &\Phi_{12} & -\Phi_{21} & -\Phi_{22} \\
(\kappa_{11}^-B_1+\bm{i} k_x C_1^T)\Phi_{11} &(\kappa_{21}^-B_1+\bm{i} k_x C_1^T)\Phi_{12} &  -(\kappa_{12}^+B_2+ \bm{i} k_x C_2^T)\Phi_{21}&  -(\kappa_{22}^+B_2+\bm{i} k_x C_2^T)\Phi_{22}
\end{bmatrix}.
\end{equation}
}
To ensure only trivial solutions for $\Re{s} >0$, the coefficients
$\boldsymbol{\delta}$ must vanish, and thus we require the {\it determinant condition}
\begin{align}\label{eq:determinat_conditon}
 \mathcal{F}(s, k_x):=\det\left(\mathcal{C}(s,  k_x)\right) \ne 0, \quad \forall\Re{s} >0.   
\end{align}
We will now formulate an algebraic definition  of stability equivalent to Lemma \ref{def:well-posedness},  for the  coupled problem,  \eqref{pde1}--\eqref{initial_con} and \eqref{int_con}, with piecewise constant media parameters \eqref{mp}.
\begin{lemma}\label{def:determinant_condition}
The solutions of the elastic wave equation \eqref{pde1} with piecewise constant media parameters \eqref{mp}  and the interface condition  \eqref{int_con} are not stable in any sense if for some $k_x \in \mathbb{R}$ and $s \in \mathbb{C}$ with $\Re{s} >0$, we have 
$$
\mathcal{F}(s, k_x):=\det\left(\mathcal{C}(s,  k_x)\right) =0.
$$
\end{lemma}
The {\it determinant condition} is defined for all $s$ with $\Re{s} >0$. The case when $\Re{s} =0$ would correspond to time-harmonic and important interface wave modes, such as Stoneley waves \cite{Gustafsson2013,Kreiss2012}. 

The energy stability in Theorem \ref{Thm_energy} states that   the  coupled problem,  \eqref{pde1}--\eqref{initial_con} and \eqref{int_con}, with piecewise constant media parameters \eqref{mp} conserves energy. Therefore, similar to Remark \ref{remark:imaginaryroots}, the roots $s$ of $\mathcal{F}(s, k_x)$ must be zero or purely imaginary, i.e.  $s\in \mathbb{C}$ with $\Re{s} =0$. 
%
%
We conclude that all nontrivial and stable interface wave modes, such as Stoneley waves, that solve $\mathcal{F}(s, k_x)=0$, must have purely imaginary roots, $s = \bm{i}\xi$ with $\xi \in \mathbb{R}$. A main objective of the present work is to determine how the purely imaginary roots $s = \bm{i}\xi$ will move in the complex plane  when the PML is introduced.

\revv{The homogeneity  property of the  determinant  $\mathcal{F}(s, k_x) = \det(\mathcal{C}(s,  k_x))$ will be important in the following analysis.}
To determine the homogeneity property of $\mathcal{F}(s, k_x)$, we may evaluate the corresponding determinant of $\mathcal{C}(s,  k_x)$. We have the following result. 
 \begin{theorem}\label{thm_homogeneous}
 In piecewise constant elastic media, the determinant  $\mathcal{F}(s, k_x) = \det\left(\mathcal{C}\left(s, k_x\right)\right)$ given in \eqref{eq:determinat_conditon}  is homogeneous of degree two.
 \end{theorem}
 \begin{proof}
 \sw{First we recall that the roots $\kappa_{ij}^{\pm}$ in \eqref{eq:kappa_roots} are homogeneous of degree one and the eigenfunctions $\Phi_{ij}$ in \eqref{eq:eigenfunctions} are  homogeneous of degree zero.}
 Consider the modified boundary matrix $\mathcal{C}_1(s,  k_x)$ where we have multiplied the first two rows of $\mathcal{C}(s,  k_x)$  by $s \ne 0$, that is
{
\footnotesize
\begin{equation}\label{F_0_C1}
\mathcal{C}_1(s,  k_x)=\begin{bmatrix}
s\Phi_{11} &s\Phi_{12} & -s\Phi_{21} & -s\Phi_{22} \\
(\kappa_{11}^-B_1+\bm{i} k_x C_1^T)\Phi_{11} &(\kappa_{21}^-B_1+\bm{i} k_x C_1^T)\Phi_{12} &  -(\kappa_{12}^+B_2+ \bm{i} k_x C_2^T)\Phi_{21}&  -(\kappa_{22}^+B_2+\bm{i} k_x C_2^T)\Phi_{22}
\end{bmatrix}.
\end{equation}
}
By inspection, every element of $\mathcal{C}_1(s,  k_x)$ is homogeneous of degree one. Therefore the determinant $\det(\mathcal{C}_1(s,  k_x))$ of the $4\times4$ matrix $\mathcal{C}_1(s,  k_x)$, using cofactor expansion, must be homogeneous of degree four.
Note that
$$
\mathcal{C}_1(s,  k_x)=
\mathcal{K}(s)\mathcal{C}(s,  k_x),
\quad 
\mathcal{K}(s)
=
\begin{bmatrix}
s & 0 &0 & 0\\
0 & s &0 & 0\\
0 & 0 &1 & 0\\
0 & 0 &0 & 1
\end{bmatrix}.
$$
Using the properties of the determinant of products of matrices we have 
$$
\det(\mathcal{C}_1(s,  k_x)) = 
\det\left(
\mathcal{K}(s)
\right)\det(\mathcal{C}(s,  k_x))=
s^2 \det(\mathcal{C}(s,  k_x)) = s^2\mathcal{F}(s,  k_x).  
$$
Since $\det(\mathcal{C}_1(s,  k_x))$ is homogeneous of degree four, therefore the determinant $\mathcal{F}(s,  k_x)$ is homogeneous of degree two.
 \end{proof}

\section{The perfectly matched layer}\label{sec:pml_model}
We consider the elastic wave equation \eqref{pde1} with the interface conditions \eqref{int_con}. 
Let the Laplace transform, in time, of $\mathbf{u}\left(x,y, t\right)$ be defined by
{
\begin{equation}
\widehat{\mathbf{u}}(x,y,s)  = \int_0^{\infty}e^{-st}{\mathbf{u}}\left(x,y,t\right)\text{dt},  \quad s = a + \bm{i}b, \quad \Re{s} = a > 0.
\end{equation}
}
We consider  a setup where the PML is included in  the $x$-direction only.
Without loss of generality, we assume that we are only interested in the solution in the left half-plane $x\leq 0$. To absorb outgoing waves, we introduce a PML outside the left half-plane and require that the material properties are invariant in $x$ in PML. 

To derive the PML model, we Laplace transform \eqref{pde1} in time, and obtain
\begin{equation}\label{lap1}
\rho_i s^2 \mb{\widehat u_i} = \f{\p}{\p x}\left(A_i\f{\p{\mb{\widehat u}_i}}{\p x}\right)+\f{\p}{\p y}\left(B_i\f{\p\mb{\widehat u}_i}{\p y}\right)+\f{\p}{\p x}\left(C_i\f{\p\mb{\widehat u}_i}{\p y}\right)+\f{\p}{\p y}\left(C^T_i\f{\p\mb{\widehat u}_i}{\p x}\right),\quad (x,y) \in \Omega_i, \quad \Re{s} >0.
\end{equation}
Note that we have tacitly assumed homogeneous initial data.
Next, we consider \eqref{lap1} in the transformed coordinate $(\widetilde x, y)$, such that
\begin{equation}\label{eqn_Sx}
\f{d\widetilde x}{dx}=1 + \f{\sigma(x)}{\alpha+s}=:\Sx.
\end{equation}
Here, $\sigma(x)\geq 0$ is the damping function and $\alpha\geq 0$ is the complex frequency shift (CFS) \cite{Kuzuoglu_and_Mittra}.
For all $s \in \mathbb{C}$ with $\Re{s} >0$, we have $\Sx \ne 0$ and $1/\Sx \ne 0$, and the smooth complex coordinate transformation \cite{Chew_and_Weedon},
\begin{equation}\label{complex_scaling}
\f{\p}{\p x} \to \f{1}{\Sx}\f{\p}{\p x}.
\end{equation}
The PML model in Laplace space is
\begin{equation}\label{lap2}
s^2 \rho_i\mb{\widehat u}_i = \f{1}{\Sx}\f{\p}{\p x}\left(\f{1}{\Sx}A_i\f{\p\mb{\widehat u}_i}{\p x}\right)+\f{\p}{\p y}\left(B_i\f{\p\mb{\widehat u}_i}{\p y}\right)+\f{1}{\Sx}\f{\p}{\p x}\left(C_i\f{\p\mb{\widehat u}_i}{\p y}\right)+\f{\p}{\p y}\left(C^T_i\f{1}{\Sx}\f{\p\mb{\widehat u}_i}{\p x}\right),\quad i=1,2,
\end{equation}
with the transformed interface conditions 
\begin{equation}\label{int_lap_pml}
\mb{\widehat u}_1=\mb{\widehat u}_2,\quad B_1\f{\p\mb{\widehat u}_1}{\p y}+C_1^T\f{1}{\Sx}\f{\p\mb{\widehat u}_1}{\p x}=B_2\f{\p\mb{\widehat u}_2}{\p y}+C_2^T\f{1}{\Sx}\f{\p\mb{\widehat u}_2}{\p x}.
\end{equation}
Choosing the auxiliary variables
\begin{align*}
\mb{\widehat v}_i=\f{1}{s+\sigma+\alpha}\f{\p \mb{\widehat u}_i}{\p x}, \quad \mb{\widehat w}_i=\f{1}{s+\alpha}\f{\p \mb{\widehat u}_i}{\p y},\quad
 \mb{\widehat q}_i=\f{\alpha}{s+\alpha} \mb{\widehat u}_i,
\end{align*}
 we invert the Laplace transformed equation \eqref{lap2} and obtain the PML model in physical space,
 {
 \small
\begin{equation}\label{eq:PML_physical_space}
\begin{split}
\rho_i\left(\frac{\p^2\mb{u}_i}{\p t^2}+ \sigma\f{\p \mb{u}_i}{\p t} -\sigma\alpha (\mb{u}_i-\mb{q}_i)\right)&=\f{\p}{\p x}\left(A_i\f{\p\mb{u}_i}{\p x}+C_i\f{\p\mb{u}_i}{\p y} - \sigma A_i\mb{v}_i\right)+\f{\p}{\p y}\left(B_i\f{\p\mb{u}_i}{\p y}+C_i^T\f{\p\mb{u}_i}{\p x} + \sigma B_i\mb{w}_i\right),\\
\f{\p \mb{v}_i}{\p t}&=-(\sigma+\alpha)\mb{v}_i+\f{\p\mb{u}_i}{\p x}, \\
\f{\p \mb{w}_i}{\p t}&=-\alpha\mb{w}_i+\f{\p\mb{u}_i}{\p y}, \\
\f{\p \mb{q}_i}{\p t}&=\alpha(\mb{u}_i-\mb{q}_i).
\end{split}
\end{equation}
}
\noindent
Similarly, inverting the Laplace transformed interface conditions \eqref{int_lap_pml} for the PML model gives
\begin{equation}\label{int_pml}
 \mb{u_1}=\mb{u_2}, \quad B_1\f{\p\mb{u_1}}{\p y}+C_1^T\f{\p\mb{u_1}}{\p x}+\sigma B_1\mb{w_1}=B_2\f{\p\mb{u_2}}{\p y}+C_2^T\f{\p\mb{u_2}}{\p x}+\sigma B_2\mb{w_2}. 
\end{equation}

In the absence of the PML, $\sigma =0$, the above model problem is energy-stable in the sense of Theorem \ref{Thm_energy} for all elastic material parameters. When $\sigma >0$, however,  the coupled PML model \eqref{eq:PML_physical_space}-\eqref{int_pml} is asymmetric with auxiliary differential equations. Thus, a similar energy-stability cannot be established in general. To analyse the stability properties of the PML model in a piecewise constant elastic medium, we use the mode analysis discussed in Section \ref{sec:mode_analysis} to prove that  exponentially growing wave modes are not supported. In \ref{app:1}, we derive a more general model with PML included in both $x$ and $y$-directions.


\section{Stability analysis of the PML model}\label{sec:stability_pml}
The stability analysis of the PML will mirror directly the mode analysis described in Section \ref{sec:mode_analysis}. We will split the analysis into two parts: plane wave analysis for the Cauchy PML problem and normal mode analysis for the interface wave modes.

\subsection{Plane waves analysis}\label{sec:planewave_pml}
We now investigate the stability of body wave modes in the PML in the whole real plane  $(x,y) \in \mathbb{R}^2$ with constant medium parameters. 
We consider constant PML damping $\sigma >0$ and uniformly constant coefficients medium parameters
$$
\rho_i = \rho, \quad A_i = A, \quad B_{i} = B, \quad C_i = C,
$$
for $\Omega_i$, $i = 1, 2$, that is there is no discontinuity of material parameters at the interface at $y =0$, \kd{and we do not consider $y=0$ as an interface}. 

Consider the wave-like solution
\begin{align}\label{eq:plane_wave_pml}
\mathbf{u}\left(x, y, t\right) = \mathbf{u}_0 e^{st+\bm{i}\left( k_x x  + k_y y \right)}, \quad \mathbf{u}_0 \in \mathbb{R}^8, \quad k_x, k_y, x, y \in \mathbb{R}, \quad t\ge 0, 
\end{align}
where $s \in \mathbb{C}$ is to be determined and relates to the stability property of the PML model. 
The PML model \eqref{eq:PML_physical_space}  is not stable if there are nontrivial solutions $\mathbf{u}$ of the form  \eqref{eq:plane_wave_pml}  with $\Re{s} > 0$.
An  $s$ with a positive real part, $\Re{s} >0$, corresponds to a plane wave solution with exponentially growing amplitude. A stable system does not admit such wave modes.

We consider the normalised wave vector $\mathbf{K} = (k_1, k_2)$, with $\sqrt{k_1^2 + k_2^2} = 1$ and  the normalised variables
$$
\revv{\lambda = \frac{s}{|\mathbf{k}|},   \quad \epsilon = \frac{\sigma}{|\mathbf{k}|},  \quad \nu = \frac{\alpha}{|\mathbf{k}|}, \quad \Sx\left(\lambda, \epsilon, \nu\right)= 1 +\frac{\epsilon}{\lambda + \nu},}
$$
\revv{where $|\mathbf{k}| = \sqrt{k_x^2 + k_y^2}. $}
Thus, if there are $\Re{\lambda} > 0$, the PML is unstable.

We insert the plane wave solution \eqref{eq:plane_wave_pml} in the PML and obtain the dispersion relation
{
\begin{align}\label{eq:PML_dispersion_relation}
&F_\epsilon(\lambda, \mathbf{K}):=F\left(\lambda,  \frac{1}{\Sx\left(\lambda, \epsilon, \nu\right)}k_1, k_2\right) = 0,
\end{align}
}
where the function $F(\lambda, \mathbf{K})$ is defined by \eqref{eq:dispersion_orthotropic} and \eqref{eq:dispersion_isotropic}. 
The scaled eigenvalue $\lambda$ is a root of the complicated nonlinear dispersion relation $F_\epsilon(\lambda, \mathbf{K})$  for the PML and defined in \eqref{eq:PML_dispersion_relation}.
\sw{Note that $\epsilon \to 0$ implies a sufficiently small PML damping $\sigma \to 0$ for a fixed frequency $|\mathbf{k}|>0$ or a sufficiently large frequency $|\mathbf{k}| \to \infty$ for a fixed PML damping $\sigma>0$.}
When the PML damping vanishes, $\epsilon =0$  we have $\Sx =1$, and $F_0(\lambda, \mathbf{K}) \equiv F(\lambda, \mathbf{K})$.  As shown in Section \ref{sec:dispersion_relation}, the roots of $F(\lambda, \mathbf{K})$ are purely imaginary and correspond to the body wave modes propagating in a homogeneous elastic medium.
When the PML damping is present $\epsilon >0$, the roots $\lambda$ can be difficult to determine. However, standard perturbation arguments yield the following well-known result \cite{duru_kreiss_2012,BECACHE2003399,APPELO2006642}.
\begin{theorem}[Necessary  condition for stability]\label{theo:high_frequency_stability}
Consider the constant coefficient PML, with $\epsilon >0$, $\nu \ge 0$. 
 Let the elastic medium be described by the phase velocity  $\mathbf{V}_p= ({V}_{px}, {V}_{py})$ and the group velocity $\mathbf{V}_g = ({V}_{gx}, {V}_{gy})$ \sw{defined} in \eqref{KVSV}.
If ${V}_{px}{V}_{gx} < 0$, \sw{for any  wave vector $\mathbf{K} = (k_1, k_2)$ with $\sqrt{k_1^2 + k_2^2} = 1$,} then at sufficiently high frequencies, \sw{$\epsilon \to 0$}, there  are corresponding  unstable wave modes with $\Re{\lambda} > 0$.
\end{theorem}

For the elastic subdomains $\Omega_i$, $i = 1, 2$, we will consider only media parameters where the necessary \emph{geometric stability condition}, ${V}_{px}{V}_{gx} \ge 0$,  is satisfied and there are no growing modes for the Cauchy PML problem.  In particular, it can be shown for isotropic elastic materials that body wave modes inside the PML are asymptotically stable  for all frequencies \cite{Duru2012JSC,Duru2014K}. In many anisotropic elastic materials the geometric stability condition and the complex frequency shift $\alpha > 0$ will ensure the stability of plane wave modes for all frequencies \cite{APPELO2006642}.
\sw{
\begin{remark}
    Note that the roots $\lambda$ can be expressed as $\lambda = \lambda_0 + \lambda_1\epsilon + O(\epsilon^2)$ such that $|\lambda_1|<\infty$ for all $\mathbf{K} = (k_1, k_2)$ with $\sqrt{k_1^2 + k_2^2} = 1$ and $\lambda_0$ are the purely imaginary roots of  $F(\lambda, \mathbf{K})=0$ for the undamped system. So if $|\mathbf{k}| \to \infty \iff \epsilon \to 0$ for a fixed $\sigma >0$ then we have $\lambda \to \lambda_0$. By Definition \ref{def:well-posedness_and_stability} the PML Cauchy problem is well-posed for all $\sigma >0$.  However, when the \emph{geometric stability condition} is violated we have $\lambda_1 >0$,  the PML is unstable and will  support plane wave solutions \eqref{eq:plane_wave_pml} with bounded exponentially growing amplitude in time.
\end{remark}
}
Next, we will characterize the stability of interface wave modes in the PML.
\subsection{Stability analysis of interface wave modes}
As above, we assume  constant PML damping $\sigma \ge 0$ and piecewise  constant elastic media parameters with a planar interface at $y = 0$.  We Laplace transform \eqref{eq:PML_physical_space}--\eqref{int_pml} in time, perform a Fourier transformation in the spatial variable $x$ of \eqref{eq:PML_physical_space}--\eqref{int_pml} and eliminate all PML auxiliary variables. 
We have
\begin{equation}\label{fou1}
\rho_i s^2 \mb{\fou u}_i = -\widetilde k_x^2A_i\mb{\fou u}_i+B_i\f{\mathrm{d}^2\mb{\fou u}_i}{\mathrm{d} y^2}+\bm{i}\widetilde k_x\left(C_i+C^T_i\right)\f{\mathrm{d}\mb{\fou u}_i}{\mathrm{d} y},\quad i=1,2,
\end{equation}
where $\widetilde k_x=k_x/\Sx$. The Laplace-Fourier transformed interface conditions  are 
\begin{equation}\label{int_fou}
\mb{\fou u_1}=\mb{\fou u_2},\quad B_1\f{\mathrm{d}\mb{\fou u_1}}{\mathrm{d} y}+\bm{i}\widetilde k_xC_1^T\mb{\fou u_1}=B_2\f{\mathrm{d} \mb{\fou u_2}}{\mathrm{d} y}+\bm{i}\widetilde k_x C_2^T\mb{\fou u_2},\quad y=0.
\end{equation}
Note the similarity between \eqref{fou1}--\eqref{int_fou} and \eqref{pde1_lap}--\eqref{int_con_Lap}; the only difference is that we have replaced $k_x$ with $\widetilde k_x$ and $\boldsymbol{\phi}_i$ with $\mb{\fou u}_i$.
When the PML damping vanishes $\sigma =0$, we have $\Sx\equiv 1$ and  $\widetilde k_x \equiv k_x$. In this case, the PML model \eqref{fou1}--\eqref{int_fou} is equivalent to the original equation \eqref{pde1_lap}--\eqref{int_con_Lap}, and \eqref{fou1} is the Laplace-Fourier transformations  of equation \eqref{pde1}. 

We seek modal solutions to \eqref{fou1} in the form 
\begin{equation}\label{gel_sol}
\mb{\fou u}_i=\boldsymbol{\Phi}_i e^{\kappa y}, \quad \boldsymbol{\Phi}_i \in \mathbb{C}^2,\quad i=1,2.
\end{equation}
Substituting \eqref{gel_sol} into \eqref{fou1}, we obtain 
\begin{equation}\label{int_sys}
\left(s^2 I + \mathcal{P}_i(\widetilde k_x, \kappa)\right)\boldsymbol{\Phi}_i = 0,\quad i=1,2,
\end{equation}
where 
\begin{align*}
\mathcal{P}_i(\widetilde k_x, \kappa)&=\widetilde k_x^2 A_i - \kappa^2 B_i-\bm{i}\widetilde k_x \kappa (C_i+C_i^T),\quad i=1,2.
\end{align*}
The existence of nontrivial  solutions to \eqref{int_sys} requires that
\begin{align}\label{eq:characteristics_f_pml}
&F_i\left(s,  \widetilde k_x, \kappa\right):= \det\left({s^2} I +  \mathcal{P}_i(\widetilde k_x, \kappa)\right) = 0,\quad i=1,2.
\end{align}
As above, we note that if we set $\kappa = \bm{i}k_y$ in $F_i\left(s,  k_x, \kappa\right)$, we get exactly the same PML dispersion relation \eqref{eq:PML_dispersion_relation} for the Cauchy problem.
Again, note also the close similarity between \eqref{eq:characteristics_f} and \eqref{eq:characteristics_f_pml}.
The roots, $\kappa = \widetilde{\kappa}_{ij}^{\pm}$, of the characteristic function $F_i\left(s,  \widetilde{k}_x, \kappa\right)$ are 
\begin{align}\label{eq;kappa_tilde}
\widetilde{\kappa}_{ij}^{-}(s, k_x)=\kappa_{ij}^{-}(s, \widetilde{k}_x), \quad \widetilde{\kappa}_{ij}^{+}(s, {k}_x) = \kappa_{ij}^{+}(s, \widetilde{k}_x), \quad j = 1, 2.
\end{align}
\revv{Note that $\kappa_{ij}^{\pm}(s,k_x)$ are the roots of the characteristic function for the undamped system,  when the PML damping vanishes, that is $\sigma =0$ with $S_x =1$ and $\widetilde{\kappa}_{ij}^{\pm}(s,k_x)$ are the roots when the PML is present, that is  $\sigma \ne 0$ with  $S_x \ne 1$. However, the roots $\kappa_{ij}^{\pm}(s,k_x)$ and  $\widetilde{\kappa}_{ij}^{\pm}(s,k_x)$ are related via the identity \eqref{eq;kappa_tilde}.}

For the proceeding analysis to  directly mirror the mode analysis  discussed in section \ref{sec:normal_modes_analysis_interface},  we will need the sign consistency between $\Re{\kappa_{ij}^{\pm}}$ and
$\Re{\widetilde{\kappa}_{ij}^{\pm}}$. \revv{To begin, we will prove that the sign consistency holds, that is for $\Re{s}=a > 0$, $\sigma \ge 0$ and $\alpha\ge 0$ we  have
\begin{align}
   \sign{\left(\Re{\kappa_{ij}^{\pm}}\right)}= \sign{\left(\Re{\widetilde{\kappa}_{ij}^{\pm}}\right)}. 
\end{align}
}
\revv{We will follow the standard procedure, by first determining $\sign{\left(\Re{\widetilde{\kappa}_{ij}^{\pm}}\right)}$ at a point in the positive complex plane, $\Re{s}>0$, then use continuity arguments to extend the result to the entire complex plane.}
\revv{ 
We begin with the Lemma which characterizes the roots for positive real $s>0$, that is $0<s \in \mathbb{R}$.
\begin{lemma}\label{lem:sign consistency for large positve Re(s)}
     For positive real $s>0$, that is $0<s \in \mathbb{R}$,  the roots $\widetilde{\kappa}_{ij}$ come in pairs and have non-vanishing real parts with
     $
        \sign{\left(\Re{\kappa_{ij}^{\pm}}\right)}= \sign{\left(\Re{\widetilde{\kappa}_{ij}^{\pm}}\right)}.
     $
\end{lemma}
\begin{proof}
    For real $s>0$, the PML complex coordinate transformation metric is real and positive $S_x >0$ for all $\sigma \ge 0$, $\alpha \ge 0$, and  $\widetilde k_x=k_x/\Sx \in \mathbb{R}$ is a real scaling. Since $\widetilde{\kappa}_{ij}^{\pm}(s, k_x)=\kappa_{ij}^{\pm}(s, \widetilde{k}_x)$ we must have    $
        \sign{\left(\Re{\kappa_{ij}^{\pm}}\right)}= \sign{\left(\Re{\widetilde{\kappa}_{ij}^{\pm}}\right)}.
     $
\end{proof}
}
The  following lemma, which uses a standard continuity arguments,  was first proven in \cite{Duru2014K}.
\begin{lemma}\label{lem:sign consistency}
If the PML Cauchy problem has no temporally growing modes, then for all $k_x \in \mathbb{R}$ and all $s  \in \mathbb{C}$ with $\Re{s} > 0$
the PML characteristic equation has roots $\widetilde{\kappa}_{ij}^{\pm}(s, k_x)$ with 
$$
\sign{\left(\Re{\kappa_{ij}^{\pm}}\right)}= \sign{\left(\Re{\widetilde{\kappa}_{ij}^{\pm}}\right)}.
$$
\end{lemma}
\begin{proof}
As above, we note that the roots vary continuously with $s$. Thus, if the real part of a root $\widetilde{\kappa}_{ij}$ changes sign, then for some $s$ with $\Re{s} >0$ the root must be purely imaginary, $\widetilde{\kappa}_{ij}  = \bm{i}k_y$. When $\kappa_{ij}  = \bm{i}k_y$ the PML dispersion relations \eqref{eq:PML_dispersion_relation} for the Cauchy problem and the characteristic   \eqref{eq:characteristics_f_pml} are equivalent. Therefore a purely imaginary root $\kappa_{ij}  = \bm{i}k_y$ with $\Re{s} >0$ corresponds to an exponentially growing mode for the Cauchy PML problem, which contradicts the assumption that the Cauchy PML problem has no growing wave modes. 
\end{proof}
\begin{remark}\label{rem:well-posed-pml}
\revv{
We remark that Lemma \ref{lem:sign consistency for large positve Re(s)} and Lemma \ref{lem:sign consistency} together with the {\it determinant condition} \eqref{eq:determinat_conditon} and the homogeneity property Theorem \ref{thm_homogeneous} indicate that the PML transmission problem \eqref{eq:PML_physical_space}--\eqref{int_pml} is well-posed, by Definition \ref{def:well-posedness_and_stability}. This claim will be made more precise by the stability analysis below.
}
\end{remark}
\noindent
We are, however, particularly interested in the stability of the PML transmission problem \eqref{eq:PML_physical_space}--\eqref{int_pml} which would exclude the possibility of exponentially growing interface wave modes.

{As before, recall that for each $i=1,2$ there are exactly two roots with positive real parts, $\kappa_{ij}^{+}$ and  exactly two roots with negative real parts, $\kappa_{ij}^{-}$, for all $j = 1, 2$. The sign consistency of the roots, Lemma \ref{lem:sign consistency for large positve Re(s)} and  Lemma \ref{lem:sign consistency}, $
\sign{\left(\Re{\kappa_{ij}^{\pm}}\right)}= \sign{\left(\Re{\widetilde{\kappa}_{ij}^{\pm}}\right)}
$, ensures that the perturbed roots have equal number of positive and negative roots, that is 
$\sign{\left(\Re{\widetilde{\kappa}_{ij}^{+}}\right)}>0$ and 
$\sign{\left(\Re{\widetilde{\kappa}_{ij}^{-}}\right)} < 0$ for all $\Re{s}>0$, $i,j \in {1, 2}$.}
Taking into account the boundedness condition, the general solution of \eqref{fou1} is
\begin{equation}\label{gen_sol_ped1_lap_pml}
\mb{\fou u_{1}}(y)=\delta_{11} e^{\widetilde{\kappa}_{11}^- y}\Phi_{11}+\delta_{12} e^{\widetilde{\kappa}_{12}^- y}\Phi_{12}, \quad \mb{\fou u_{2}}(y) =\delta_{21} e^{\widetilde{\kappa}_{21}^+ y}\Phi_{21}+\delta_{22} e^{\widetilde{\kappa}_{22}^+ y}\Phi_{22}.
\end{equation}
The coefficients $\boldsymbol{\delta}=[\delta_{11},\delta_{12},\delta_{21},\delta_{22}]^T$ are determined by inserting \eqref{gen_sol_ped1_lap_pml} into the interface conditions \eqref{int_con_Lap}. We have the following equation
\begin{equation}\label{Fs_pml_9}
\mathcal{C}(s,  \widetilde{k}_x)\boldsymbol{\delta}=\mb{0},
\end{equation}
where 
{
\footnotesize
\begin{equation}\label{F}
\mathcal{C}(s,  \widetilde{k}_x)=\begin{bmatrix}
\Phi_{11} &\Phi_{12} & -\Phi_{21} & -\Phi_{22} \\
(\widetilde{\kappa}_{11}^-B_1+\bm{i}\widetilde k_x C_1^T)\Phi_{11} &(\widetilde{\kappa}_{21}^-B_1+ \bm{i} \widetilde k_x C_1^T)\Phi_{12} &  -(\widetilde{\kappa}_{12}^+B_2+ \bm{i}\widetilde k_x C_2^T)\Phi_{21}&  -(\widetilde{\kappa}_{22}^+B_2+ \bm{i} \widetilde k_x C_2^T)\Phi_{22}
\end{bmatrix}.
\end{equation}
}
Using the determinant condition given in Definition \ref{def:determinant_condition}, we formulate a stability condition for the PML in a piecewise constant elastic medium.
\begin{lemma}[Stability condition]\label{lem:determinant_condition}
The solution to the PML model \eqref{fou1} with piecewise constant material parameters \eqref{mp} and interface condition \eqref{int_fou} is not stable in any sense if for some $k_x \in \mathbb{R}$ and $s \in \mathbb{C}$ with $\Re{s} >0$, the determinant vanishes, 
$$
\mathcal{F}(s, \widetilde{k}_x):=\det\left(\mathcal{C}(s,  \widetilde{k}_x)\right) =0.
$$
\end{lemma}

The roots of  $\mathcal{F}(s, \widetilde{k}_x)$ are tightly connected to the roots of $\mathcal{F}(s, {k}_x)$ by the homogeneous property of  $\mathcal{F}$. As a consequence, it is enough to analyse the roots of $\mathcal{F}(s, \widetilde{k}_x)$ for the stability property of the PML model. 


\begin{theorem}\label{s0}
Let $\Fh(s, k_x)$ be a homogeneous function of degree $n\in \mathbb{Z}$. Assume that  $\Fh(s, k_x) \ne 0$  for all $\Re{s} >0$ and $k_x \in \mathbb{R}$. Let $\widetilde{k}_x = k_x/\Sx$, where $\Sx$ is the PML metric \eqref{eqn_Sx} with $\sigma\ge0$ and $\alpha\ge0$.  Then the function $\Fh(s, \widetilde{k}_x)$ has no root $s$ with positive real part, $\Re{s} > 0$.
\end{theorem}
\begin{proof}
Consider the homogeneous function $\Fh(s, k_x)$, we have
\begin{align*}
\Fh(s, \widetilde k_x)=\Fh\left(s, \frac{k_x}{\Sx}\right) = \Fh\left(\frac{s\Sx}{\Sx}, \frac{k_x}{\Sx}\right)=\left(\f{1}{\Sx}\right)^n\Fh\left(s\Sx, k_x\right).
\end{align*}
Since ${\Sx} \ne 0$ and ${1}/{\Sx} \ne 0$, we must have 
\begin{align*}
\Fh(s, \widetilde k_x)=0 \iff \Fh\left(\widetilde{s}, k_x\right) =0, \quad \widetilde{s} = s\Sx.
\end{align*}
Assume that $s = a + \bm{i}b$ with $a > 0$, we have
\begin{align*}
&\Re{\widetilde{s}} =  \left(a + \left( \frac{a\left(a+\alpha\right) + b^2}{{|s + \alpha|^2}}\right)\sigma\right) \ge a>0.
\end{align*}
Thus if $\Fh\left(\widetilde{s}, k_x\right) =0$ then 
$\widetilde{s}$ with $\Re{\widetilde{s}} > 0$ is a root. This will contradict the assumption that  $\Fh\left({s}, k_x\right) \ne 0 $ for all $\Re{s} >0$. We conclude that for $s = a + \bm{i}b$ with $a > 0$, we must have $\Fh\left(\widetilde{s}, k_x\right) \ne 0$ for all $\sigma \ge 0$ and $\alpha \ge 0$.
\end{proof}

 \revv{By Theorem \ref{thm_homogeneous} we know that the  determinant  $\mathcal{F}(s, k_x) = \det(\mathcal{C}(s,  k_x))$ is a homogeneous function of degree two. Thus, we can now state the result that shows that exponentially growing waves modes are not supported by the PML  in a discontinuous elastic medium.}
  \begin{theorem}
Consider the PML \eqref{fou1} in a discontinuous elastic medium  with the interface condition \eqref{int_fou} at $y = 0$. Let $\mathcal{F}(s, k_x)$ be the homogeneous function given in \eqref{eq:determinat_conditon}. If $\mathcal{F}(s, k_x) \ne 0$  for all $\Re{s} >0$ and $k_x \in \mathbb{R}$ and the PML Cauchy problem has no temporally growing modes,    then there are no growing interface wave modes in the PML. That is $\mathcal{F}(s, \widetilde{k}_x) \ne 0$ for all $\Re{s} >0$ and $k_x \in \mathbb{R}$.
\end{theorem}
\begin{proof}
The proof is identical to the proof of Theorem \ref{s0} with degree of homogeneity $n = 2$.
\end{proof}
 %
 %
The following theorem states that interface wave modes are dissipated by the PML, \sw{i.e., $\mathcal{F}(s, \widetilde{k}_x) =0$ implies $\Re{s} \le 0$ for all $k_x \in \mathbb{R}$, $\alpha \ge 0$ and $\sigma \ge 0$.}
 \begin{theorem}\label{theorem:dissipation_of_interface_waves}
Consider the PML  model problem \eqref{fou1} in a discontinuous elastic medium  with the interface condition \eqref{int_fou} at $y = 0$. If the PML Cauchy problem has no temporally growing modes then all stable interface wave modes, that solve $\mathcal{F}(s, k_x) =0$ for all $k_x \in \mathbb{R}$ with $s = \bm{i}\xi$, are dissipated by the PML. 
\end{theorem}
\begin{proof}
We will split the proof into two cases, for $\alpha = 0$ and $\alpha >0$.

Consider
\begin{align*}
\mathcal{F}(s, \widetilde k_x)=0 \iff \mathcal{F}\left(\widetilde{s}, k_x\right) =0, \quad \widetilde{s} = s\Sx=\f{\alpha+s+\sigma}{\alpha+s}s, \quad \alpha, \sigma  \ge 0.
\end{align*}
Since $\mathcal{F}\left({s}_0, k_x\right)$ has purely imaginary roots ${s}_0=\bm{i}\xi$, we must have
\begin{equation}\label{eqni}
\f{\alpha+s+\sigma}{\alpha+s}s=\bm{i}\xi,
\end{equation}
for some $\xi\in\mathbb{R}$. 
Thus, if $\alpha=0$, then $s=-\sigma+\bm{i}\xi$ and $\Re{s}=-\sigma <0$, for $\sigma>0$.

When $\alpha > 0$, we consider
\begin{equation*}
\f{\alpha+s+\sigma}{\alpha+s}s=\bm{i}\xi \iff s^2 + (\alpha + \sigma -\bm{i}\xi)s - \bm{i}\alpha\xi = 0.
\end{equation*}
If $\xi =0$, then the roots are $s =0$ and $s = -(\alpha + \sigma) < 0$, for $\alpha > 0$, $\sigma >0$. Clearly the real parts of the roots are non-positive.
If $\xi \ne 0$, then the roots are given by
$$
s = -\frac{(\alpha + \sigma -\bm{i}\xi)}{2} \pm \frac{1}{2}\sqrt{(\alpha + \sigma -\bm{i}\xi)^2 + \bm{i}4\alpha\xi}.
$$
The real parts of the two roots are
$$
\Re{s} = -\frac{(\alpha + \sigma)}{2}\pm \frac{1}{2\sqrt{2}}\sqrt{(\alpha + \sigma)^2 - \xi^2 + \sqrt{((\alpha + \sigma)^2 - \xi^2)^2 + 4 \xi^2(\alpha - \sigma)^2}}.
$$
We note that the root with a negative sign has a negative real part,
$$
\Re{s} = -\frac{(\alpha + \sigma)}{2}- \frac{1}{2\sqrt{2}}\sqrt{(\alpha + \sigma)^2 - \xi^2 + \sqrt{((\alpha + \sigma)^2 - \xi^2)^2 + 4 \xi^2(\alpha - \sigma)^2}} < 0.
$$
For the other root with a positive sign, we have
$$
\Re{s} = -\frac{(\alpha + \sigma)}{2}+ \frac{1}{2\sqrt{2}}\sqrt{(\alpha + \sigma)^2 - \xi^2 + \sqrt{((\alpha + \sigma)^2 - \xi^2)^2 + 4 \xi^2(\alpha - \sigma)^2}}.
$$
If we assume that $\Re{s} > 0$ for $\alpha > 0$, $\sigma >0$ and $\xi \in \mathbb{R}$, then this implies that
$$
(\alpha + \sigma) < \frac{1}{\sqrt{2}}\sqrt{(\alpha + \sigma)^2 - \xi^2 + \sqrt{((\alpha + \sigma)^2 - \xi^2)^2 + 4 \xi^2(\alpha - \sigma)^2}}.
$$
Squaring both sides of the inequality gives
$$
(\alpha + \sigma)^2 + \xi^2 < \sqrt{((\alpha + \sigma)^2 - \xi^2)^2 + 4 \xi^2(\alpha - \sigma)^2}.
$$
Squaring both sides again and simplifying further yields
$$
(\alpha+\sigma)^2 < (\alpha-\sigma)^2.
$$
This is a contradiction  since $\alpha > 0$ and $\sigma > 0$.
Thus, for $\alpha > 0$ and $\sigma > 0$, we must have $\Re{s} < 0$.
The roots are moved further by the PML into the stable complex plane.
\end{proof}

\section{Numerical Experiments}\label{sec:num_exp}

In this section, we present extensive numerical examples to verify the stability analysis performed in the previous sections and demonstrate the absorption properties of the PML model for the elastic wave equation. We will consider a sequence of numerical experiments with increasing model complexities. We will end the section by simulating elastic wave scattering using the \kd{Marmousi} model \cite{marmousi2} defined by heterogeneous and discontinuous elastic medium in a 2D domain $\Omega \subset \mathbb{R}^2$.

For the spatial discretisation, we use the SBP finite difference operators with fourth-order accurate interior stencil \cite{Mattsson2004}. The boundary conditions and material interface conditions are imposed weakly by the penalty technique \cite{Duru2014V, Duru2014,Carpenter1994} such that a discrete energy estimate is obtained when the damping vanishes. For details on the SBP discretisation and stability for the undamped problem, we refer the reader to \cite{Duru2014, Duru2014V}. We discretise in time using the classical fourth-order accurate Runge-Kutta method with stable explicit time steps, and small enough so that the error is dominated by spatial discretisation. The Julia code to solve all the examples can be found online on a Github Repository\footnote[1]{Github link to the repository: \url{https://github.com/Balaje/Summation-by-parts}}.

\subsection{Two Layers}

We consider the elastic wave equation in a two-layered medium $\Omega_1 \cup \Omega_2$, where $\Omega_1 = [0,4\pi]^2$ and $\Omega_2 = [0,4\pi] \times [-4\pi, 0]$. The material properties in each layer is either isotropic or orthotropic  elastic solid. For the isotropic case, we use the material properties $\rho_1 = 1.5, \mu_1 = 4.86, \lambda_1 = 4.8629$ in $\Omega_1$, and $\rho_2 = 3, \,\mu_2 = 27, \,\lambda_2 = 26.9952$ in $\Omega_2$. For the orthotropic material property, we choose $\rho_1 = 1,\, c_{11_{1}} = 4,\, c_{12_{1}} = 3.8, \, c_{22_{1}} = 20$ and $c_{33_{1}} = 2$ in $\Omega_1$, and the material properties in $\Omega_2$ are chosen as $\rho_2 = 0.25$ and $c_{ij_2} = 4c_{ij_1}$ for $i,j = 1,2$.

For initial conditions, we set the initial displacements \sw{in both spatial directions} as the Gaussian 
\begin{align}\label{eq:Gaussian_initial_data}
   u_{ij}=e^{-20((x-x_s)^2 + (y-y_s)^2)},\quad i,j=1,2,
\end{align}
centered at $(x_s, y_s) =(2\pi, 1.6\pi)$ and zero initial conditions for the velocity field and all auxiliary variables. We impose the characteristic boundary conditions at the left boundary $x=0$, the bottom boundary $y=-4\pi$, and the top boundary $y=4\pi$. Outside the right boundary, at $x=4\pi$, we use a PML $[4\pi, 4.4\pi] \times [-4\pi, 4\pi]$ closed by the characteristic boundary condition at the PML boundaries. Because of the PML, the boundary conditions must be modified as 
\begin{align}
Z_{1y}\f{\p\mb{u_1}}{\p t}+B_1\f{\p\mb{u_1}}{\p y}+C_1^T\f{\p\mb{u_1}}{\p x}+B_1\sigma\mb{w_1} + \sigma Z_{1y}(\mb{u_1}-\mb{q_1})&=0, \quad y=4\pi,\label{TwoLayerBCtop}\\
Z_{ix}\f{\p\mb{u}_i}{\p t}-A_i\f{\p\mb{u}_i}{\p x}-C_i\f{\p\mb{u}_i}{\p y}+A_i\sigma\mb{v}_i&=0, \quad x=0,\ \revv{{i}=1,2},\label{TwoLayerBCleft}\\
Z_{ix}\f{\p\mb{u}_i}{\p t}+A_i\f{\p\mb{u}_i}{\p x}+C_i\f{\p\mb{u}_i}{\p y}-A_i\sigma\mb{v}_i&=0, \quad x=4.4\pi,\ \revv{{i}=1,2},\label{TwoLayerBCright}\\
Z_{2y}\f{\p\mb{u_2}}{\p t}-B_2\f{\p\mb{u_2}}{\p y}-C_2^T\f{\p\mb{u_2}}{\p x}-B_2\sigma\mb{w_2}+\sigma Z_{2y}(\mb{u_2}-\mb{q_2})&=0, \quad y=-4\pi,\label{TwoLayerBCbottom}
\end{align}
see the derivation in \cite{Duru2014K}. 
The impedance matrices $Z_{ix}$ and $Z_{iy}$ are given by
$$
Z_{ix} =
\begin{bmatrix}
\rho_ic_{px i} && 0\\
0&& \rho_ic_{sx i} \\
\end{bmatrix}, 
\quad
Z_{iy} =
\begin{bmatrix}
\rho_ic_{sy i} && 0\\
0&& \rho_ic_{py i} \\
\end{bmatrix},
$$
and the wave speeds $c_{px i}, c_{py i}, c_{sx i}, c_{sy i}$ are defined in \eqref{elastic_wave_speeds_an}, for each layer.

The PML damping function is 
\begin{equation}\label{eq:damping_func}
\begin{split}
&\sigma\left(x\right) = \left \{
\begin{array}{rl}
0 \quad {}  \quad {}& \text{if} \quad x \le L_x,\\
\sigma_0\Big(\frac{x-L_x}{\delta}\Big)^3  & \text{if}  \quad x \ge L_x ,
\end{array} \right.
\end{split}
\end{equation}
where the damping strength is 
\begin{equation}\label{eq:damping_strength}
\sigma_0=\frac{4c_{p,\max}}{2\delta}\log\left(\frac{1}{Ref}\right).
\end{equation}
Here, $c_{p,\max}=\max(c_{p1},c_{p2})$, $c_{p1} = \max(c_{px 1}, c_{py 1})$ and $c_{p2}= \max(c_{px 2}, c_{py 2})$ are the maximum pressure wave speeds in $\Omega_1$ and $\Omega_2$, respectively. The parameter $L_x = 4\pi$ is the length of the domain,  $\delta=0.1 L_x$ is the width of the PML and $Ref=10^{-4}$ is the relative PML modeling error \revv{\cite{Liu2023}, which is the residual error arising after domain truncation}. Additionally, we choose the CFS parameter $\alpha=0.05\sigma_0$ in both subdomains.


\begin{figure}
    \includegraphics[width=0.32\textwidth]{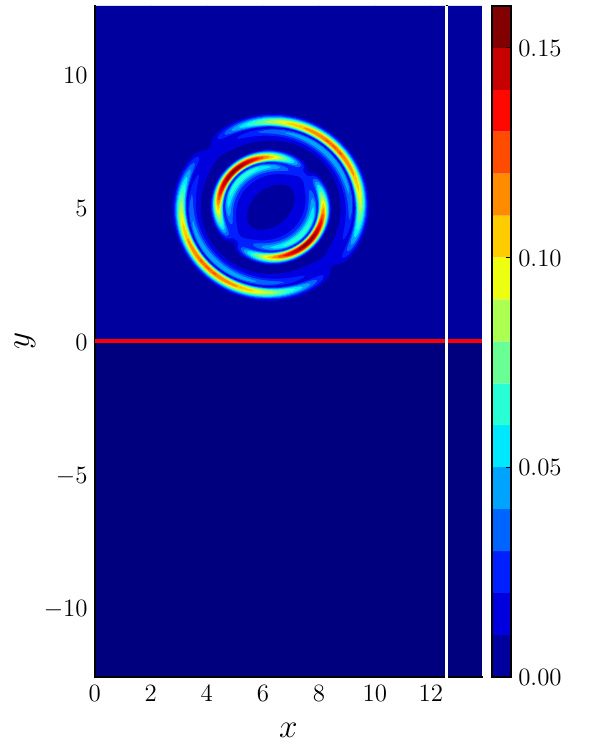}        
    \includegraphics[width=0.32\textwidth]{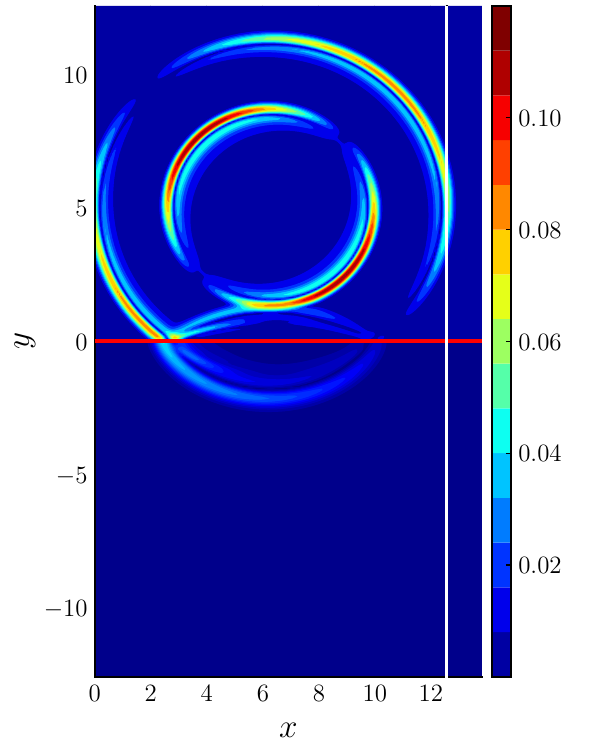}        
    \includegraphics[width=0.32\textwidth]{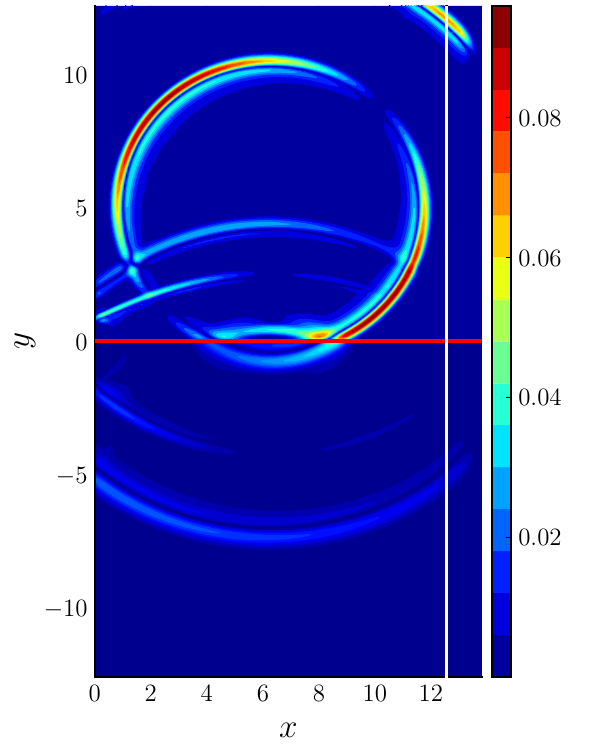}     
    \caption{The solution at three time points $t=1, 2, 3$ in a piecewise isotropic medium. The red, solid curve denotes the interface between the two layers. The region to the right of the black, dashed line is the PML.}
    \label{fig:isotropic-2-layer-cartesian}
\end{figure}
\begin{figure}
    \includegraphics[width=0.32\textwidth]{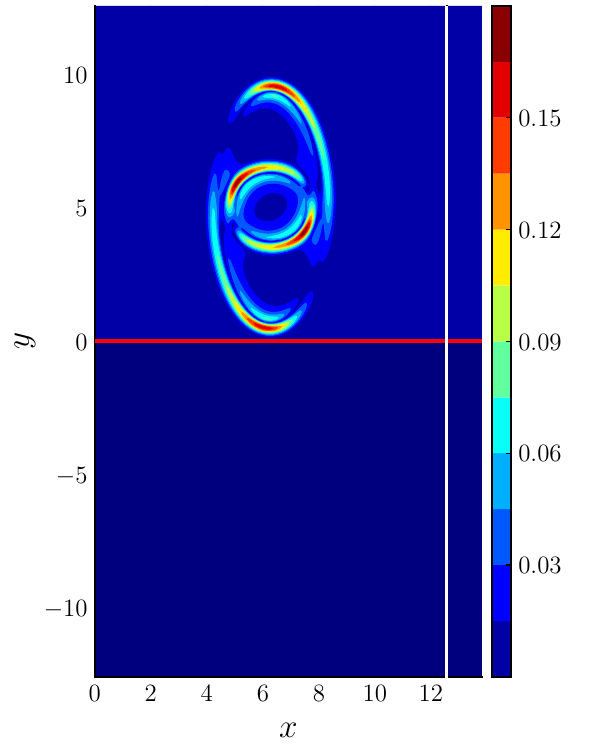}        
    \includegraphics[width=0.32\textwidth]{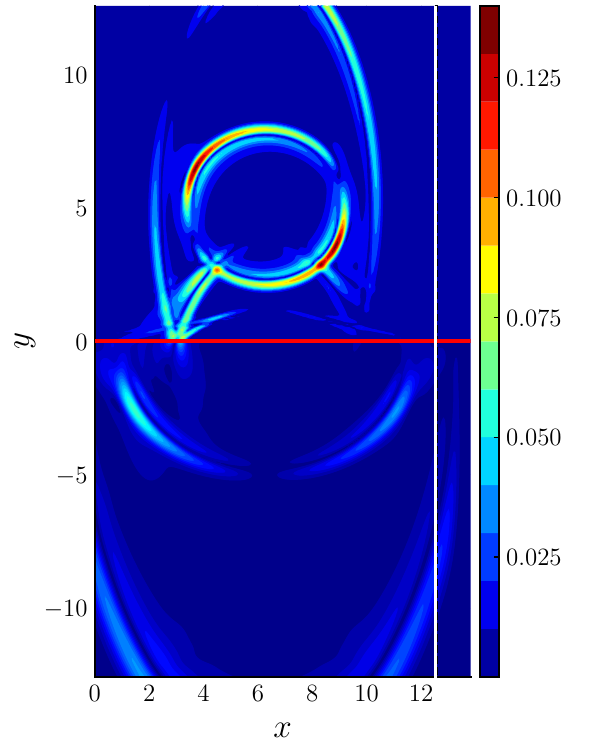}        
    \includegraphics[width=0.32\textwidth]{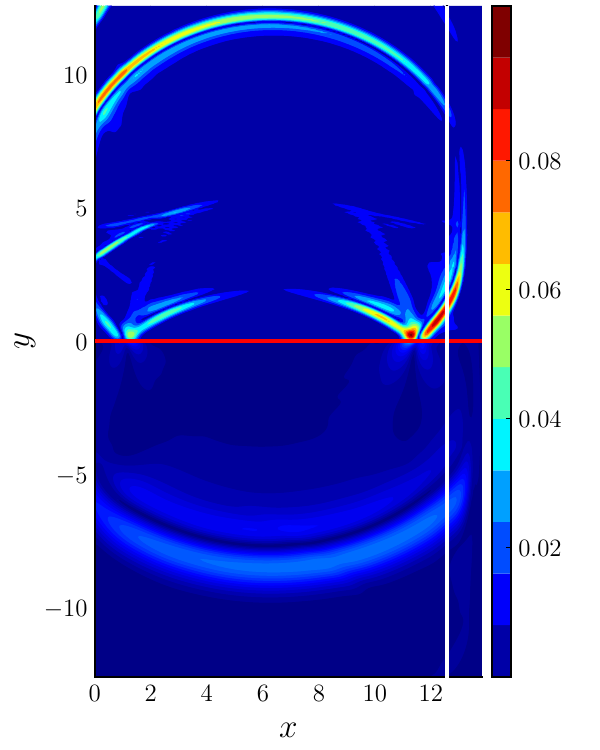}     
    \caption{Solution at three time points $t=1, 2, 5$ in a piecewise anisotropic, othrotropic medium. The red, solid curve denotes the interface between the two layers. The region to the right of the black, dashed line is the PML.}
    \label{fig:anisotropic-2-layer-cartesian}
\end{figure}

In Figures~\ref{fig:isotropic-2-layer-cartesian}~and~\ref{fig:anisotropic-2-layer-cartesian}, we plot the numerical solutions at three time points for the isotropic and anisotropic media, respectively. In both cases, the initial data is a Gaussian \eqref{eq:Gaussian_initial_data} in the top layer. At $t=1$, we observe that a wave mode propagates at the same speed in the two spatial directions in the isotropic medium but at different speeds in the anisotropic medium. At $t=2$, the elastic waves have propagated into the bottom layer in the middle panel, where the effects of discontinuity are observed. In the last panel, we observe that the waves coming into the PML are effectively absorbed without reflections. 

\begin{figure}
    \begin{subfigure}{0.5\textwidth}
        \includegraphics[width=\textwidth]{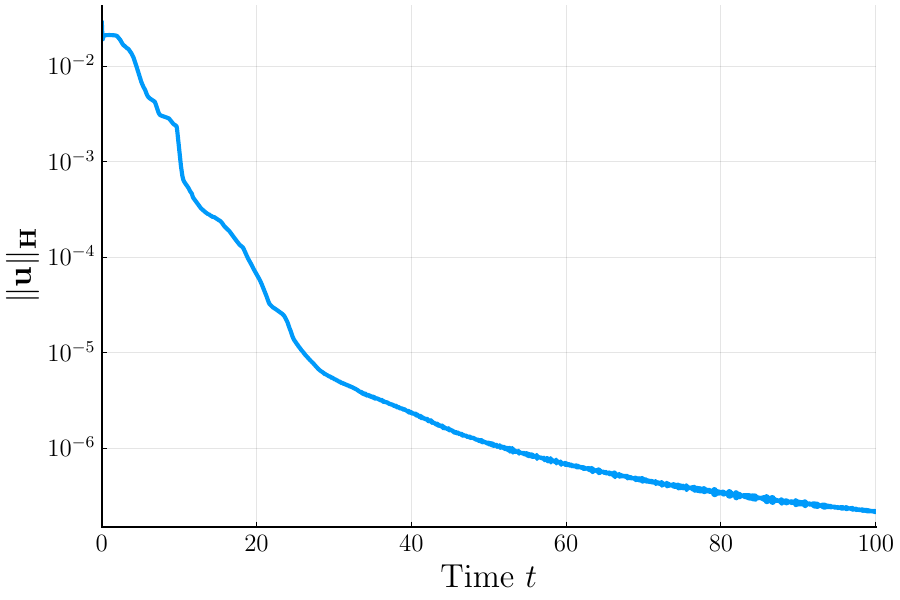}
    \end{subfigure}
    \begin{subfigure}{0.5\textwidth}
        \includegraphics[width=\textwidth]{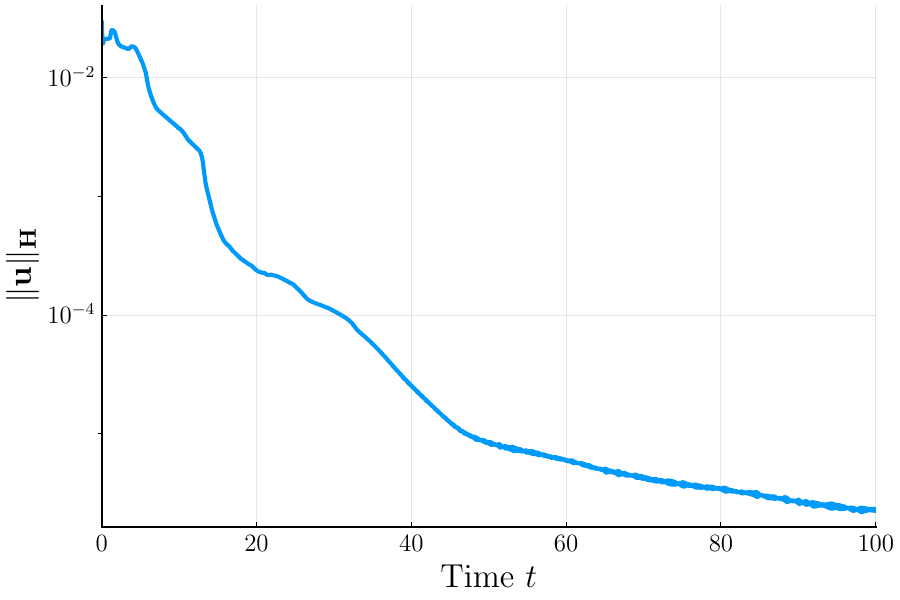}
    \end{subfigure}
    \caption{The quantity $\| \mathbf{u} \|_{H}$ with $t$ for the isotropic (left) and the orthotropic (right) media.}
    \label{fig:dispVsTime}
\end{figure}

\sloppy In Figure~\ref{fig:dispVsTime}, we plot the $l_2$-norm, $\| \mathbf{u} \|_{H} = \sqrt{\sum_{i=1}^2 \mathbf{u}_i^T \mathbf{H} \mathbf{u}_i}$, of the numerical solutions in time, where $i$ is the layer index and $\mathbf{H}$ is the discrete norm associated with the SBP operator. We observe that $\| \mathbf{u} \|_{H}$ decays monotonically in both the isotropic and anisotropic media. Note that at the final time $t=100$ the largest amplitude is about $10^{-7}$, demonstrating the numerical stability and the effectiveness of PML.

\begin{figure}
    \begin{subfigure}{0.5\textwidth}
        \includegraphics[width=\textwidth]{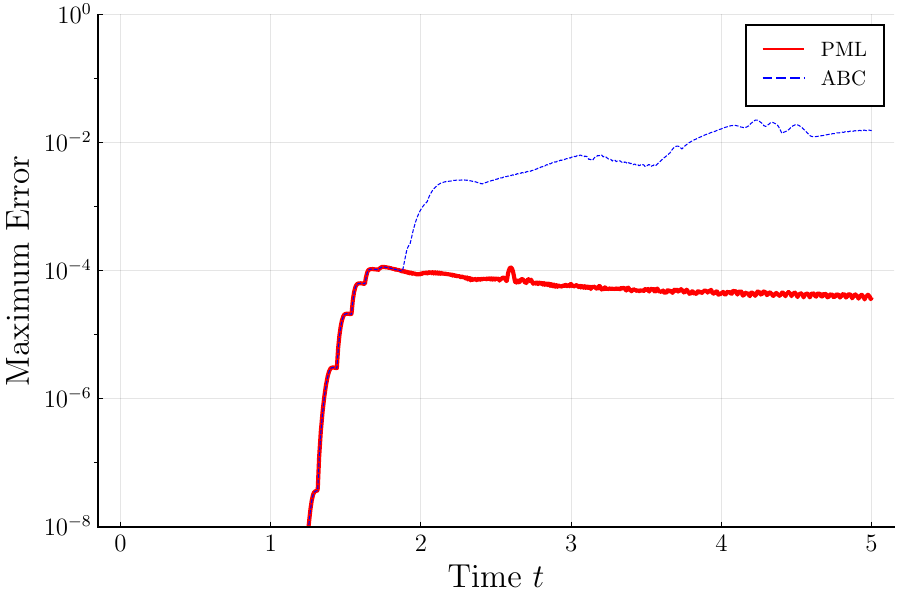}
    \end{subfigure}
    \begin{subfigure}{0.5\textwidth}
        \includegraphics[width=\textwidth]{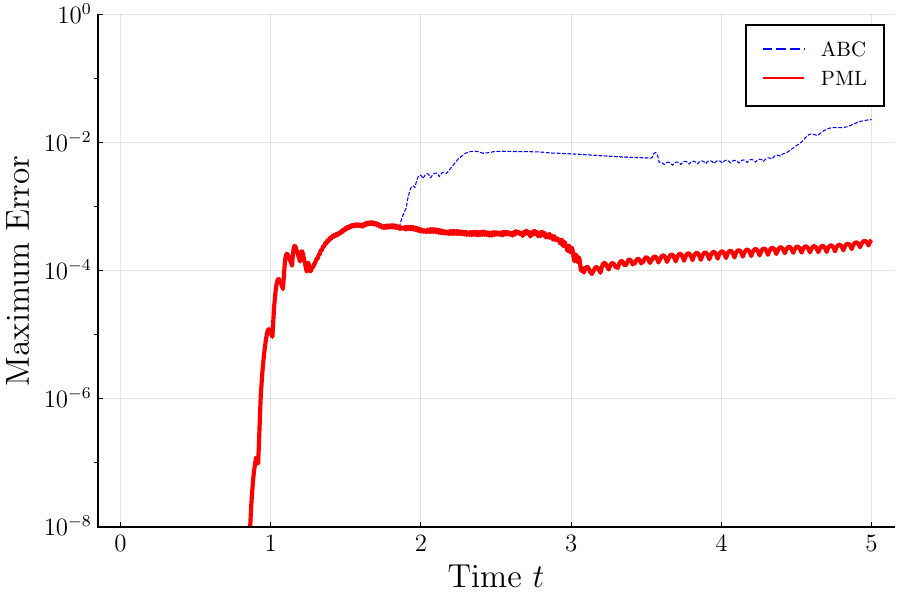}
    \end{subfigure}
    \caption{The maximum error for the isotropic (left) and the orthotropic (right) media. ABC is the absorbing boundary condition.}
    \label{fig:errVsTime}
\end{figure}

Finally, we compare the absorbing property of the PML model with the first order absorbing boundary conditions (ABC) \cite{doi:10.1061/JMCEA3.0001144}. We compute a solution in the large domain $[0,4\pi]\times[-4\pi,12\pi]$, which is the original domain extended three times in the positive $x$ direction, and regard the part of the solution in  $[0,3.6\pi]\times[-4\pi,4\pi]$ as a reference solution. 
As before, we consider the initial displacement \eqref{eq:Gaussian_initial_data}
and zero initial conditions for the velocity and the auxiliary variables. In Figure~\ref{fig:errVsTime}, we plot the PML error defined as the maximum norm of the difference between the PML solution and the reference solution, and the ABC error that is defined analogously as the maximum norm of the difference between the solution computed by using the ABC on all boundaries and the reference solution. We observe that the PML error is about two order of magnitude smaller than the ABC error in both isotropic and anisotropic media.

\subsection{Four layers}
Next we demonstrate extension of the results to multiple elastic layers.
We consider the elastic wave equation in domain  $\Omega=[0,40]\times[-40,0]$. The medium has a four-layered structure and the material parameters are summarized in Table \ref{tab_mp}. In each layer, the material property is homogeneous and isotropic.  At the interface between two adjacent layers, the material property is discontinuous, and the equations are coupled by imposing continuity of displacement and traction in the form of \eqref{int_pml}. 
At time $t=0$, we initialise the displacement fields as the Gaussian
$$
{u}_{ij} =  e^{-5((x-x_s)^2+(y-y_s)^2)},\quad {i}=1,2,3,4,\ j=1,2
$$
centered in the middle of Layer 2, that is  $(x_s, y_s) =(20, -15)$. 
\begin{table}
    \centering
    \begin{tabular}{ccccc}
        Layer & $\rho$ & $c_s$ & $c_p$ & Domain    \\
        \hline 
        1 & 1.5 & 1.8 & 3.118 & $[0,40] \times [-10,0]$ \\
        2 & 1.9 & 2.3 & 3.984 & $[0,40] \times [-20,-10]$ \\
        3 & 2.1 & 2.7 & 4.667 & $[0,40] \times [-30,-20]$ \\
        4 & 3   & 3   & 5.196 & $[0,40] \times [-40,-30]$ \\
            \hline
    \end{tabular}
    \caption{Material properties in the four layers.}
    \label{tab_mp}
\end{table}

We impose a traction free boundary condition at the top boundary $y=0$. At the left, bottom and right boundaries, we add a PML of width $\delta = 4$ which is 10\% of the computational domain in $x$. The PML in the $y$-direction introduces a new auxiliary variable in the governing equation, see Appendix A.  At the boundaries of the PML ($x=\left\{-4, 44\right\}$ and $y = -44$), we impose the characteristic boundary condition. Because of the PML, the boundary conditions must be modified as 
\begin{align}
B_1\f{\p\mb{u_1}}{\p y}+C_1^T\f{\p\mb{u_1}}{\p x}+ \sigma_x B_1\mb{w_1}&=0, \quad y=0,\label{FourLayerBCtop}\\
Z_{ix}\f{\p\mb{u}_i}{\p t}-A_i\f{\p\mb{u}_i}{\p x}-C_i\f{\p\mb{u}_i}{\p y}+A_i\sigma_x\mb{v_i} + \sigma_y Z_{ix}(\mb{u}_i - \mb{q}_i)&=0, \quad x=0,\ {i}=1,2,3,4,\label{FourLayerBCleft}\\
Z_{ix}\f{\p\mb{u}_i}{\p t}+A_i\f{\p\mb{u}_i}{\p x}+C_i\f{\p\mb{u}_i}{\p y}-A_i\sigma_x\mb{v_i} + \sigma_y Z_{ix}(\mb{u}_i - \mb{q}_i)&=0, \quad x=44,\ {i}=1,2,3,4,\label{FourLayerBCright}\\
Z_{4y}\f{\p\mb{u_4}}{\p t}-B_4\f{\p\mb{u_4}}{\p y}-C_4^T\f{\p\mb{u_4}}{\p x}-B_4\sigma_x\mb{w_4}+\sigma_x Z_{4y}(\mb{u_4}-\mb{q_4})&=0, \quad y=-40.\label{FourLayerBCbottom}
\end{align}
More precisely, on the $y$-boundaries the modified traction includes the auxiliary variable $\mb{w}$. In addition,  the time derivative in the characteristic boundary condition introduces a lower order term, see  \eqref{FourLayerBCbottom}. Similarly, on the $x$-boundaries, the modified traction includes the auxiliary variable $\mb{v}$.

%

Inside the PML of all four layers, we choose the damping functions  $\sigma_x(x)$, $\sigma_y(y)$ which are cubic monomials similar to \eqref{eq:damping_func}
and the damping strength $\sigma_0 > 0$ is given by \eqref{eq:damping_strength} with  $Ref=10^{-4}$ the relative PML modeling error. 
Here, $c_{p,max}=\max_i{c_{pi}}$ is the largest pressure wave speed $c_{pi}$ in $\Omega_i$, $i=1,2,3,4$. Additionally, we choose the CFS parameter $\alpha=0.05\sigma_0$.

We use the same spatial and temporal discretisation parameters as in the previous numerical example. In Figure \ref{fig_FourDomainsSol441}, we plot the solutions at three time points with the grid size $h=0.1$. We observe that at $t=3$, the initial Gaussian displacement field has expanded from its centre across the top three layers and the reflections at the material interfaces are clearly visible.  At $t=5$, the wave has propagated across all four layers, and has interacted with the free surface, at $y=0$, and the characteristic boundary condition, at $x=-4$. The plot at $t=9$ shows that the surface and interface waves entering the PML is effectively absorbed without reflections. 


\begin{figure}
    \begin{subfigure}{0.32\textwidth}
        \includegraphics[width=\textwidth, trim={1.2cm 0 1.5cm 0}, clip]{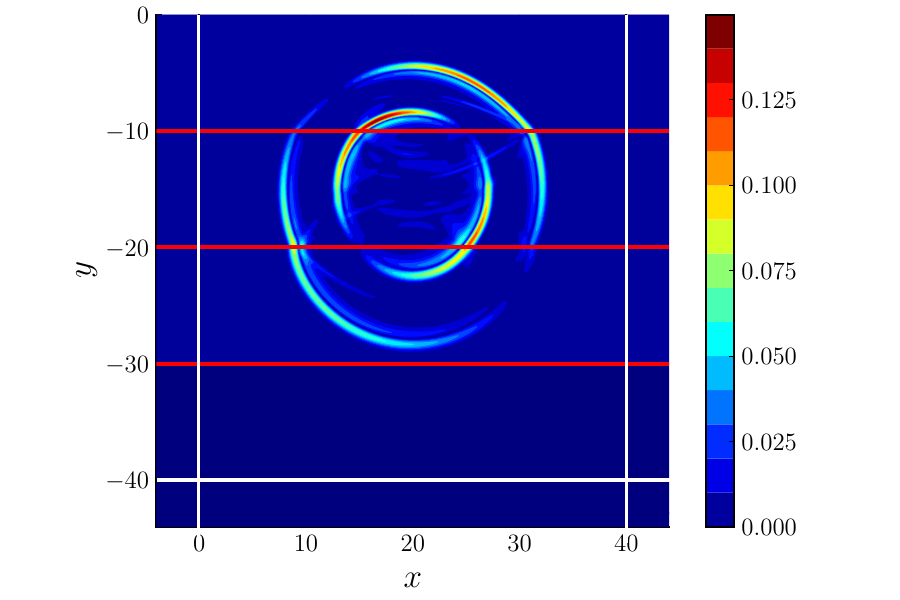}
    \end{subfigure}
    \begin{subfigure}{0.32\textwidth}
        \includegraphics[width=\textwidth, trim={1.2cm 0 1.5cm 0}, clip]{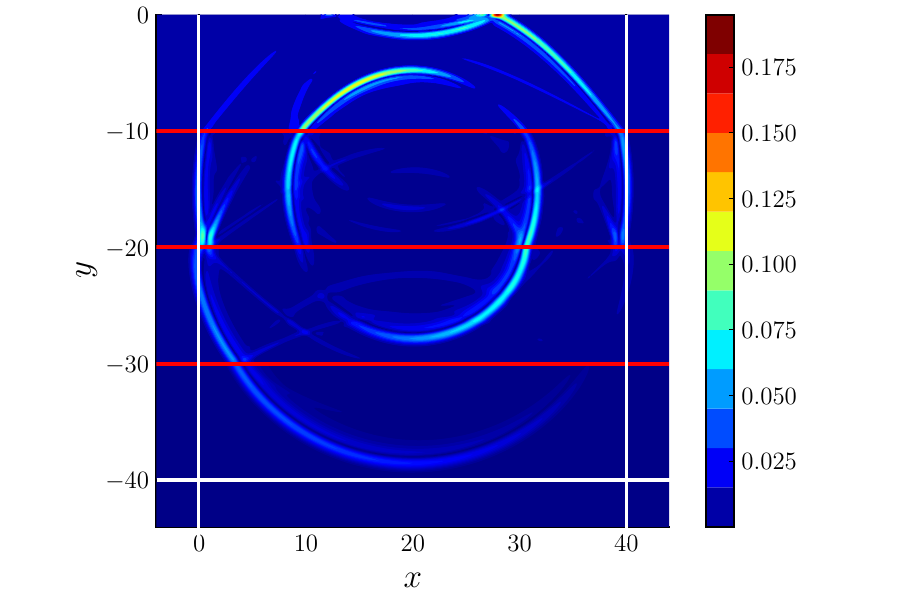}
    \end{subfigure}
    \begin{subfigure}{0.32\textwidth}
        \includegraphics[width=\textwidth, trim={1.2cm 0 1.5cm 0}, clip]{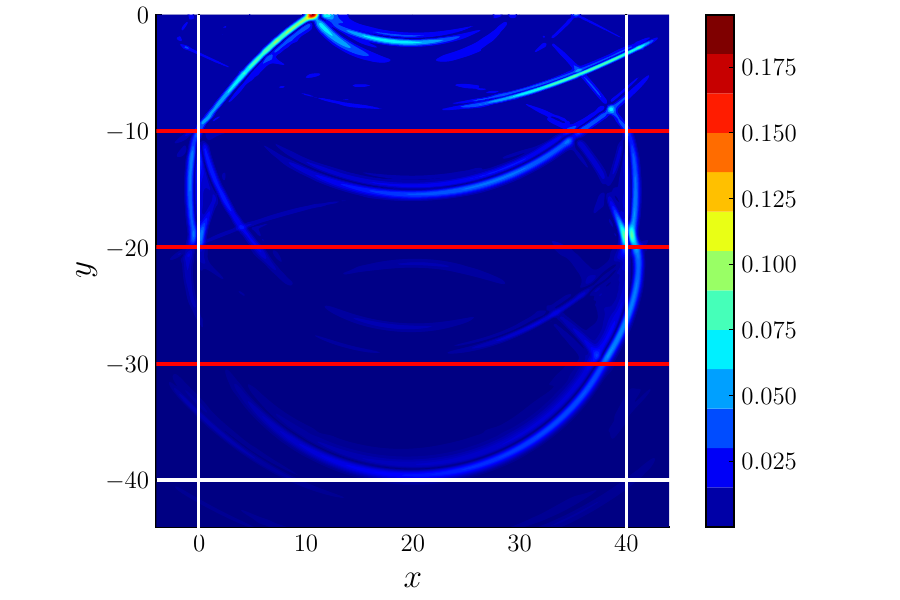}
    \end{subfigure}
    \caption{The absolute displacement $| \mathbf{u} |$ at three time points $t=3,5,9$ with Gaussian initial data and grid size $h=0.1$. The horizontal red lines indicate the material interfaces.}
    \label{fig_FourDomainsSol441}
\end{figure}

Next, we consider an example driven by seismological sources, an explosive  moment tensor point source $F=g M_0 \nabla f_{\delta}$, as the forcing in the governing equation. The moment time function $g$ and the approximated delta function $f_{\delta}$ take the form 
\begin{align*}
    g=e^{-\frac{(t-0.215)^2}{0.15}},\quad f_{\delta}=\frac{1}{2\pi\sqrt{s_1s_2}}e^{-\left(\frac{(x-20)^2}{2s_1}+\frac{(y+15)^2}{2s_2}\right)},
\end{align*}
where the parameters $s_1=s_2=0.5h$ and $M_0=1000$. We note that the peak amplitude of $F$ is located in the middle of Layer 2, that is at $(x_s, y_s) = (20, -15)$. With zero initial data for all variables, we run the simulation with $h=0.1$ 
and plot the solutions  in Figure \ref{fig_FourDomainsSol441Forcing}. 
We have similar observation as the case with a Gaussian initial data.  

\begin{figure}
    \begin{subfigure}{0.32\textwidth}
        \includegraphics[width=\textwidth, trim={1.3cm 0 1.6cm 0}, clip]{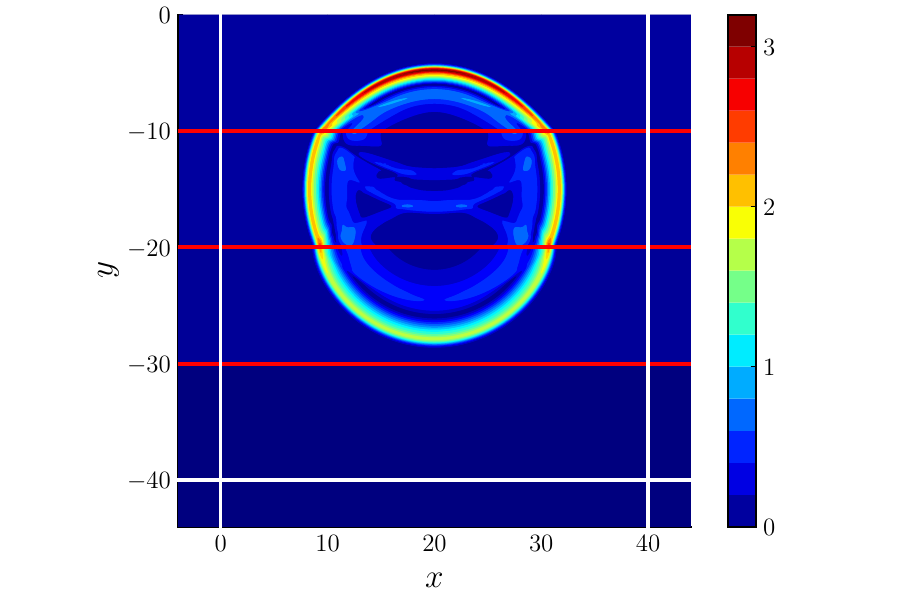}
    \end{subfigure}
    \begin{subfigure}{0.32\textwidth}
        \includegraphics[width=\textwidth, trim={1.3cm 0 1.6cm 0}, clip]{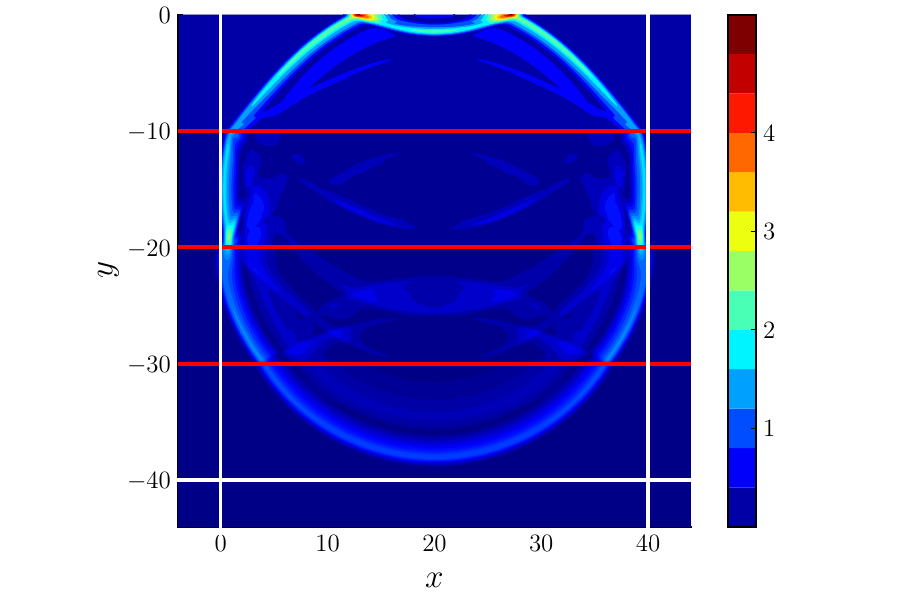}
    \end{subfigure}
    \begin{subfigure}{0.32\textwidth}
        \includegraphics[width=\textwidth, trim={1.3cm 0 1.6cm 0}, clip]{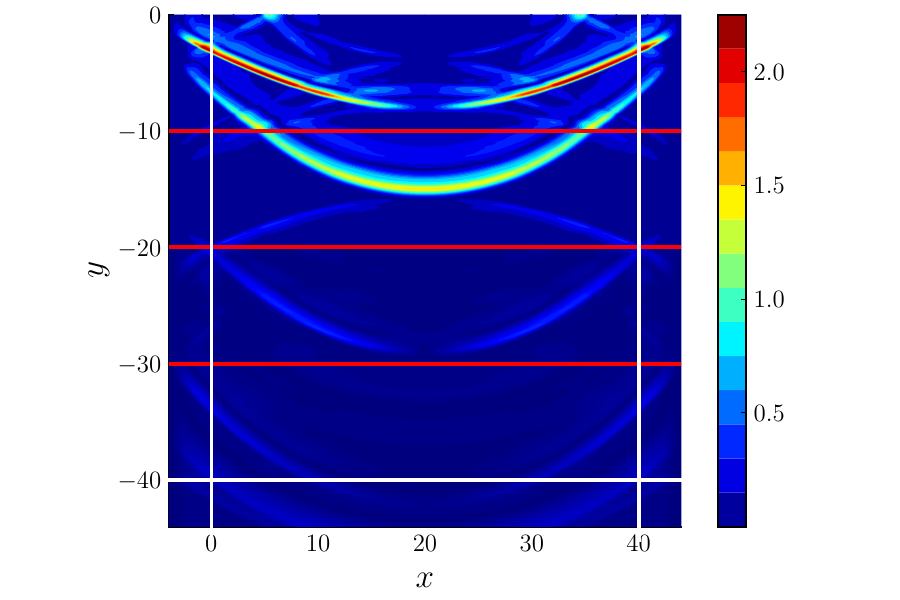}
    \end{subfigure}
    \caption{The solution at three time points $t=3,5,9$ with single point moment source and grid size $h=0.1$. The horizontal red lines indicate the material interfaces.}
    \label{fig_FourDomainsSol441Forcing}
\end{figure}
 To see the stability property of the PML, we plot $\|\mb{u}\|_H$ in time in Figure \ref{fig_FourDomainsSolEnergy}. The first plot corresponds to the case with initial Gaussian  data, and the second plot corresponds to the case with the single point moment source. It is clear that the PML remains stable after a long time $t=1000$.

\begin{figure}
    \begin{subfigure}{0.5\textwidth}
       \includegraphics[width=\textwidth]{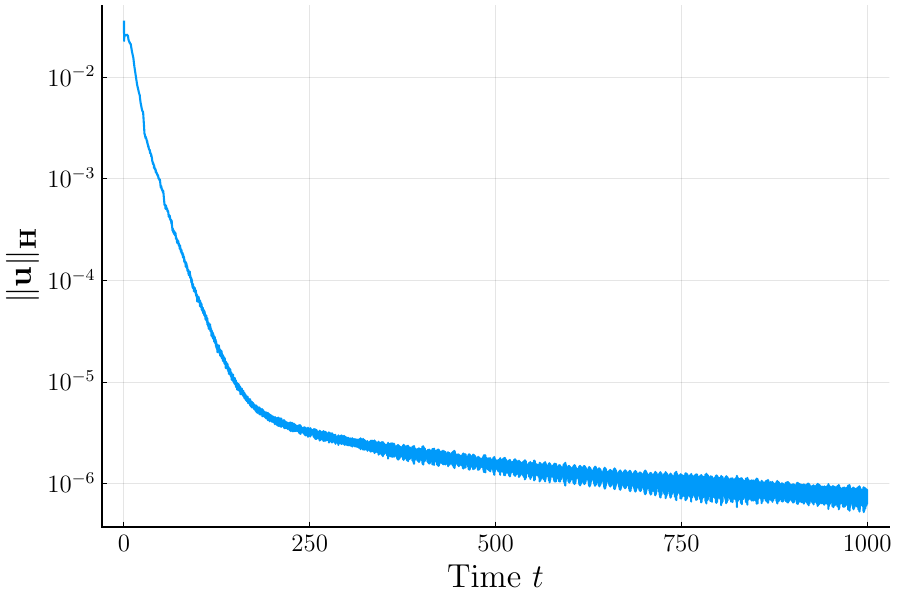}
    \end{subfigure}
    \begin{subfigure}{0.49\textwidth}
        \includegraphics[width=\textwidth]{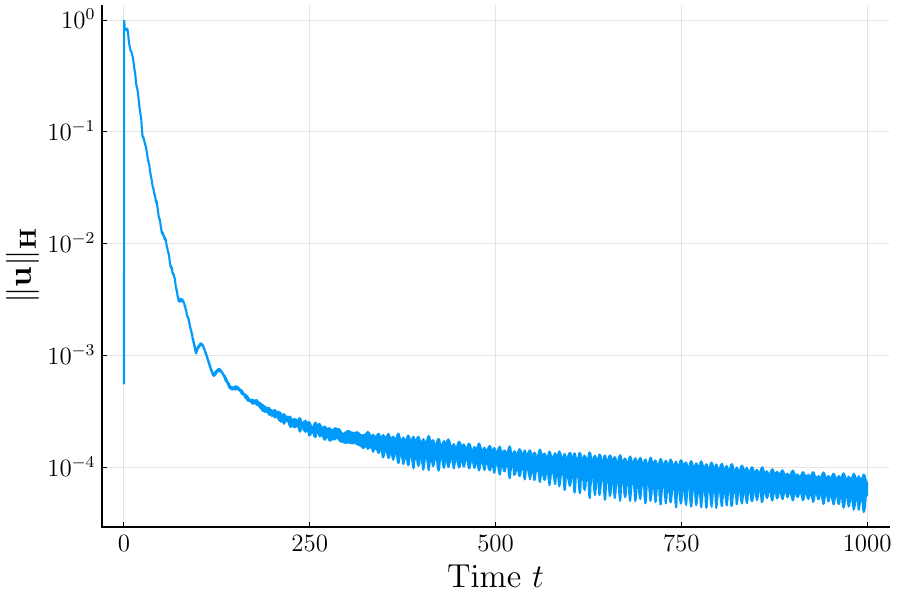}
    \end{subfigure}
    \caption{The quantity  $\|\mb{u}\|_H$ with $h=0.1$ for the Gaussian initial data (left) and the single point moment source (right).}
    \label{fig_FourDomainsSolEnergy}
\end{figure}

\subsection{Curved material interface}
In this section, we consider the elastic wave equation in a two-layered medium in $(x,y)\in [0,4\pi]\times[-4\pi,4\pi]$ with a curved material interface shown in Figure~\ref{fig:two-layer-domain} (left). The two layers are separated by a smooth Gaussian hill centered at the midpoint of the domain that is parameterised by $y = 0.8\pi e^{-10(x-2\pi)^2}$. 
We consider isotropic media with the material properties $\rho_1 = 1.5, \mu_1 = 4.86, \lambda_1 = 4.8629$ in $\Omega_1$, and $\rho_2 = 3, \,\mu_2 = 27, \,\lambda_2 = 26.9952$ in $\Omega_2$. 

We impose the characteristic boundary conditions on the left, bottom and the top boundary. We introduce a PML $[4\pi, 4.4\pi] \times [-4\pi, 4\pi]$ with the damping function \eqref{eq:damping_func}--\eqref{eq:damping_strength}, closed by the characteristic boundary condition at the PML boundaries. We set the initial displacements as the Gaussian \eqref{eq:Gaussian_initial_data} centered at $(x_s, y_s) =(2\pi, 1.6\pi)$
and zero initial conditions for the velocity and auxiliary variables. 

For the spatial discretisation, we construct a curvilinear grid in each subdomain that conforms with its boundaries by the transfinite interpolation technique. The governing PML equation is then transformed to the reference domain $[0,1]^2$ and discretised, see the details of the coordinate transformation in Appendix B, and the reference \cite{Almquist2021}. 
We use a uniform reference grid of size $h$ in both layers, which results in a conforming grid interface, 
see the coarse versions of the finite difference grids on the physical domain in Figure~\ref{fig:two-layer-domain}. The numerical tests were performed on a $201 \times 201$ uniform grid on the reference domain $[0,1]^2$.

\begin{figure}
    \begin{subfigure}{0.5\textwidth}
        \centering
        \includegraphics[width=0.7\textwidth]{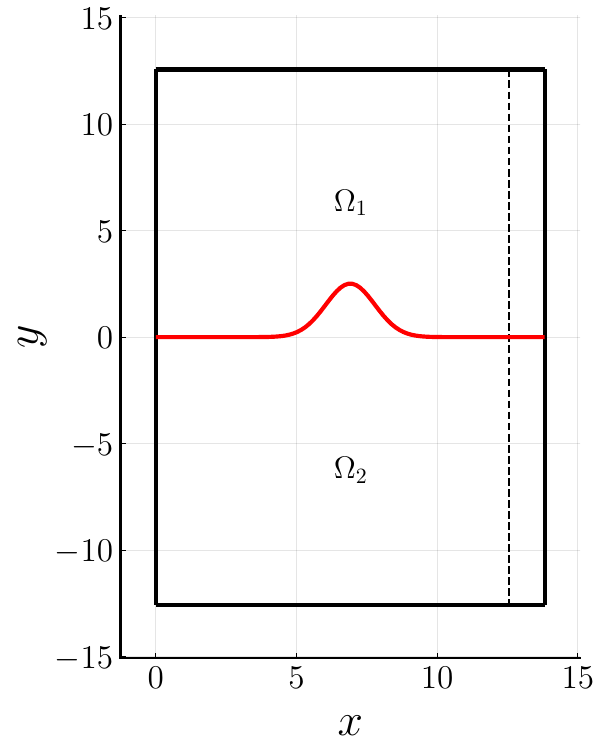}
    \end{subfigure}
    \begin{subfigure}{0.5\textwidth}
        \centering
        \includegraphics[width=0.7\textwidth]{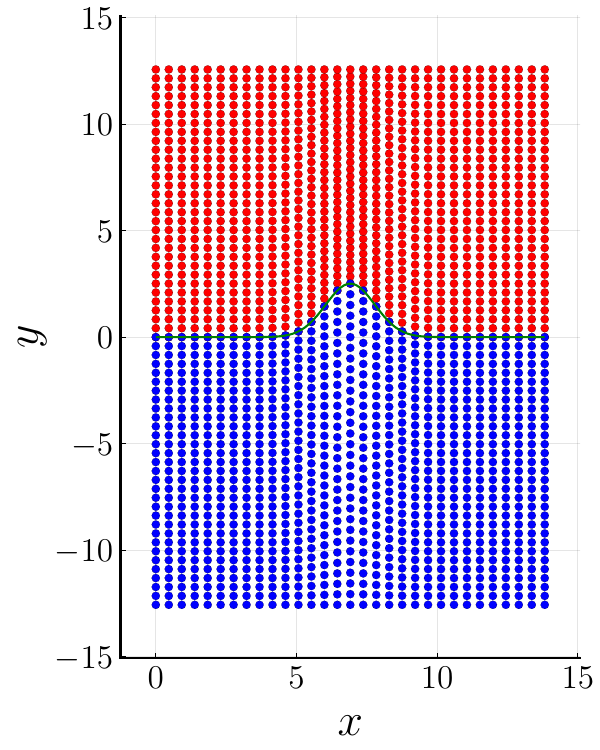}
    \end{subfigure}
    \caption{(Left) Geometry of the two-layered media along with the interface (shown in red, solid line) and the PML boundary (shown in black, dashed line). Coarse versions of an interface-conforming finite difference grid (right).}
    \label{fig:two-layer-domain}
\end{figure}

\begin{figure}
   \begin{subfigure}{0.32\textwidth}
        \includegraphics[width=\textwidth]{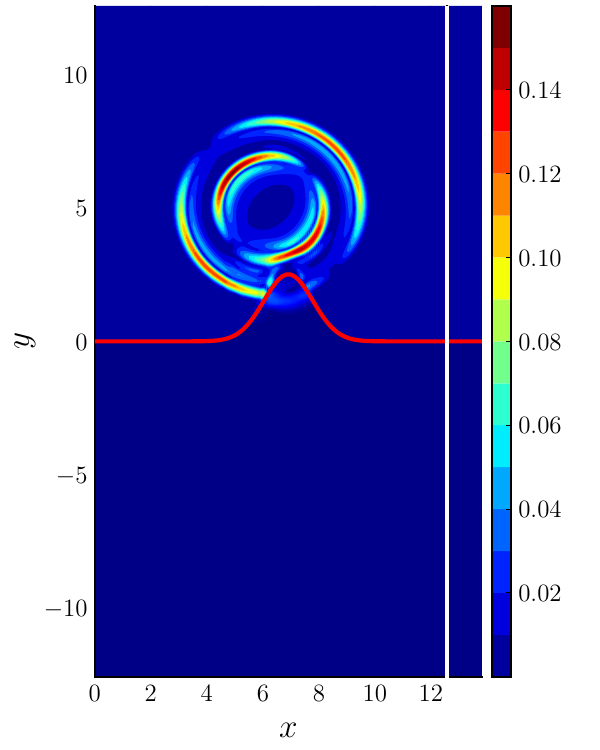}
    \end{subfigure}
    \begin{subfigure}{0.32\textwidth}
        \includegraphics[width=\textwidth]{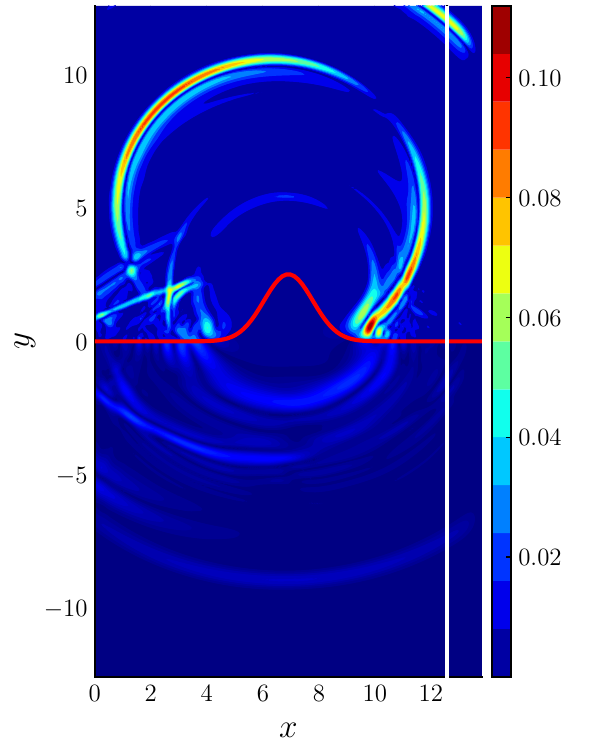}
    \end{subfigure}
    \begin{subfigure}{0.32\textwidth}
        \includegraphics[width=\textwidth]{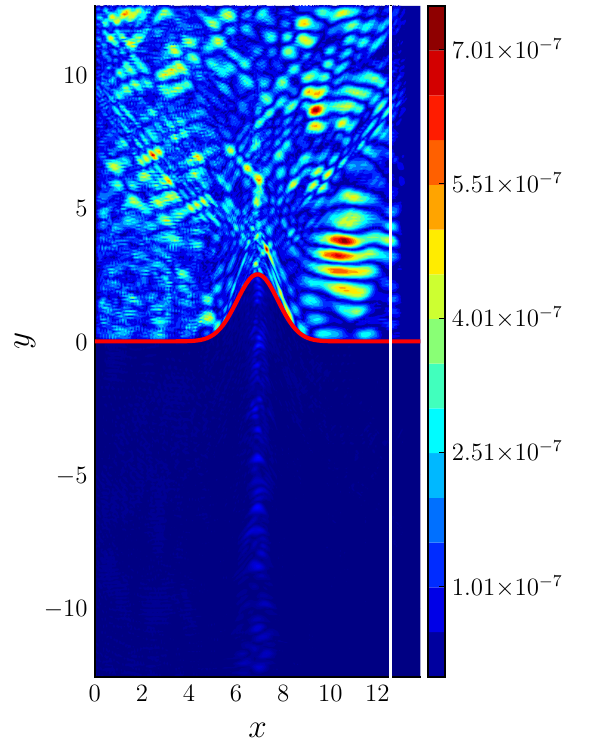}
    \end{subfigure}
    \caption{The solution at three time points $t=1, 3, 100$ in a piecewise isotropic medium. The red, solid curve denotes the interface between the two layers. The region to the right of the black, dashed line is the PML.}
    \label{fig:isotropic-curvilinear}
\end{figure}

In Figure~\ref{fig:isotropic-curvilinear}, we show snapshots of the solution to the PML. We clearly observe the effect of the curved interface from the solution-snapshots at $t=1$ and $t=3$. At $t=3$, we observe that the waves entering the PML are damped. We also show the long-time solution at $T=100$ and observe that the largest amplitude is about $10^{-7}$, demonstrating the numerical stability of PML with a curved material interface.  In Figure~\ref{fig:maxDispVsTime}, we plot  $\| \mathbf{u} \|_{H}$ in time for both  isotropic and orthotropic cases. We observe that $\| \mathbf{u} \|_{H}$ decays monotonically in both cases. Finally, we compare PML with the absorbing boundary condition in Figure~\ref{fig:PMLABCcurved}. We consider curvilinear, isotropic and orthotropic layered media with an interface conforming discretisation as shown in Figure~\ref{fig:two-layer-domain}. For the orthotropic material property, we choose $\rho_1 = 1,\, c_{11_{1}} = 4,\, c_{12_{1}} = 3.8, \, c_{22_{1}} = 20$ and $c_{33_{1}} = 2$ in $\Omega_1$, and the material properties in $\Omega_2$ are chosen as $\rho_2 = 0.25$ and $c_{ij_2} = 4c_{ij_1}$ for $i,j = 1,2$. We observe that the maximum error of the PML solution is about two magnitudes smaller in both the isotropic and orthotropic case.

\begin{figure}
    \begin{subfigure}{0.5\textwidth}
        \includegraphics[width=\textwidth]{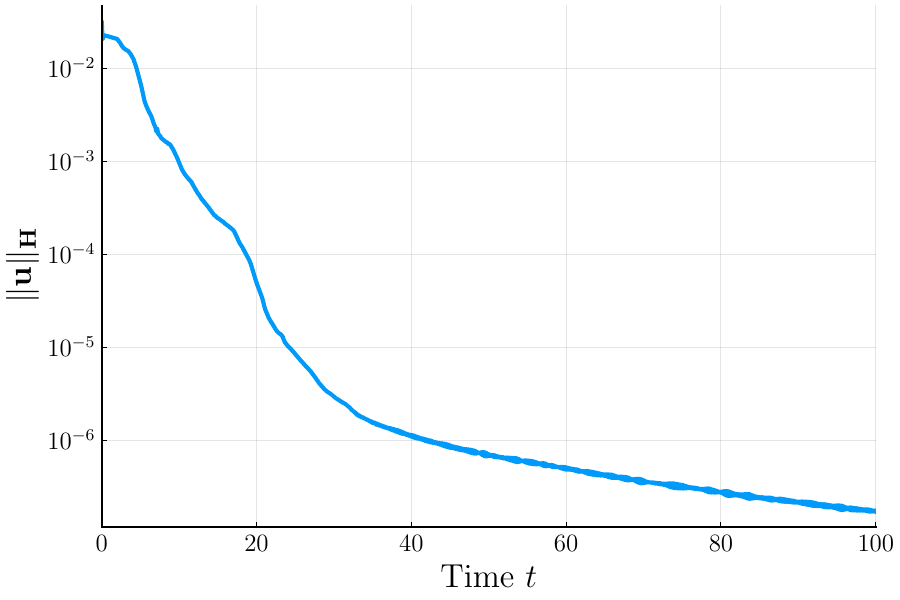}
    \end{subfigure}
    \begin{subfigure}{0.5\textwidth}
        \includegraphics[width=\textwidth]{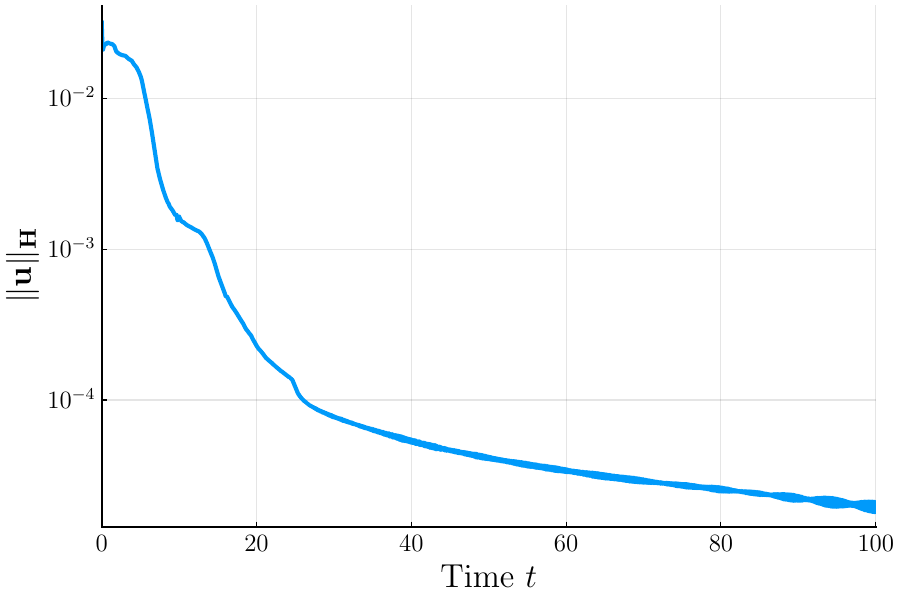}   
    \end{subfigure}
    \caption{The quantity $\| \mathbf{u} \|_{H}$ with $t$ for the conforming finite difference grid for the isotropic case (left) and the orthotropic case (right).}
    \label{fig:maxDispVsTime}
\end{figure}

\begin{figure}
    \begin{subfigure}{0.5\textwidth}
            \includegraphics[width=\textwidth]{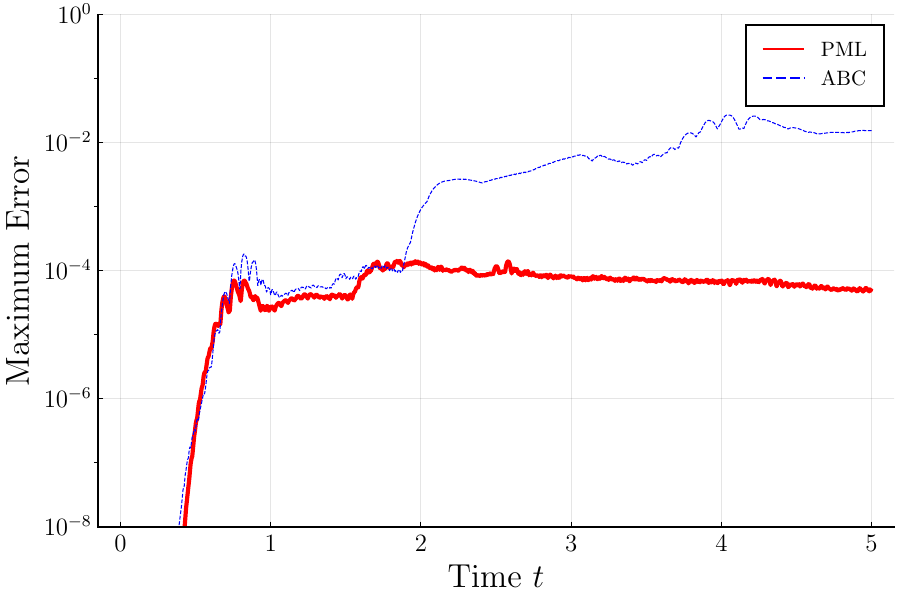}
    \end{subfigure}
    \begin{subfigure}{0.5\textwidth}
            \includegraphics[width=\textwidth]{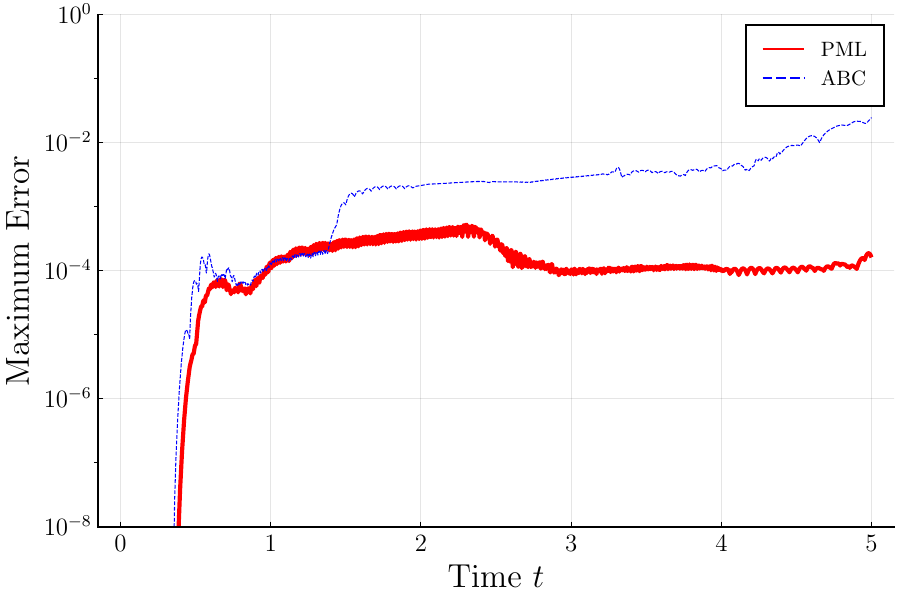}
    \end{subfigure}
    \caption{The maximum error for the isotropic (left) and the orthotropic (right) curvilinear media. ABC is the absorbing boundary condition. 
    }  
    \label{fig:PMLABCcurved}
\end{figure}



\subsection{Marmousi Model}
In this section, we consider the application of the PML to a heterogeneous and discontinuous elastic solid. We use Marmousi2 \cite{marmousi2}, a large geological dataset based upon the geology from the North Quenguela Trough in the Quanza Basin of Angola. The Marmousi2 is a fully elastic model which contains the density of the material and the $p-$and~$s-$wave speeds. In Figure~\ref{fig:s-wave-marmousi}, we show the plots of the $p-$and~$s-$wave speeds and the density in the material without the water column. 
\begin{figure}[h!]
\centering
    \begin{subfigure}{\textwidth}
        \includegraphics[width=\textwidth]{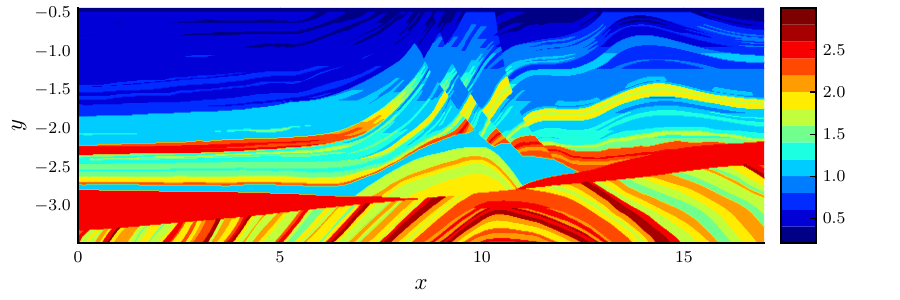}
    \end{subfigure}
    \begin{subfigure}{\textwidth}
        \includegraphics[width=\textwidth]{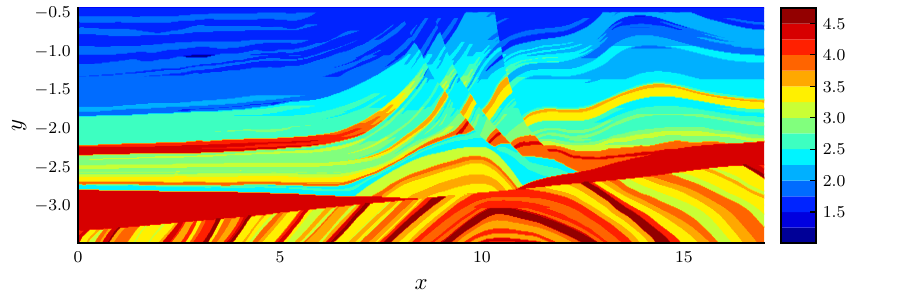}        
    \end{subfigure}
    \begin{subfigure}{\textwidth}
        \includegraphics[width=\textwidth]{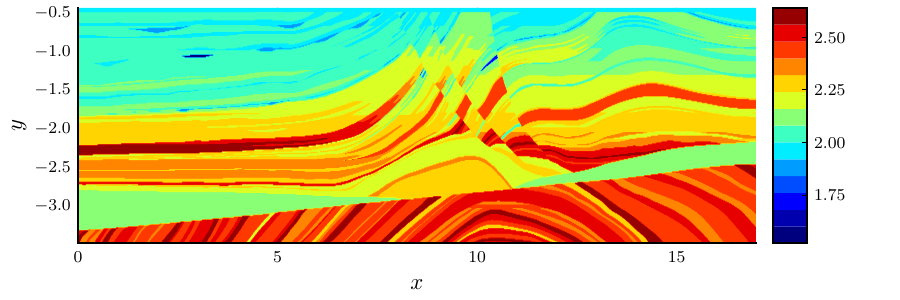}        
    \end{subfigure}
    \caption{The s-wave speed (top panel), p-wave speed (middle panel) and the density (low panel) for the Marmousi model.
    }
    \label{fig:s-wave-marmousi}
\end{figure}
The material is highly heterogeneous containing various structural elements including water, water-wet sand channels, oil-gas mixtures, salt and shale. The full computational model is defined on a rectangular domain $[0.0, 16.9864] \times [-3.4972, -0.44964]$ sampled on a $13601 \times 2801$ uniform grid. Since the primary goal of the section is to demonstrate the stability of the PML, to save computational cost, we consider the down-sampled version of the model on $1201 \times 241$ uniform grid. 

\rev{
We divide the domain into two layers by introducing an interface on a naturally-occurring discontinuity with large jump in the material parameter. 
The material interface is given by the straight line between the points $P_1= (0, -3.34)$ and $P_2 = (16.9864, -2.47)$. 
Since this interface is not parallel to the axes, we construct a curvilinear grid in each layer in the same way as in the previous example with a curved interface.   
In the reference grid, we use $1201 \times 201$ points in Layer 1 and $1201 \times 41$ points in Layer 2,  resulting in a conforming grid interface. The transformed equation and the interface conditions take the same form as the case with a curved interface, see Appendix B. As the focus is stability of the PML, we do not attempt to resolve the heterogeneous material properties within each layer, and discretise the governing equation using the standard SBP finite difference operators with variable coefficients \cite{Mattsson2012}.   

We consider the PML damping functions  $\sigma_x(x)$, which is a cubic monomial similar to \eqref{eq:damping_func}
with the PML width $\delta= 0.1L$, where $L = 16.9864$ is the length of the domain. Inside the PML, we use the original parameters from the Marmousi model. }

We consider the explosive moment tensor point sources $F = gM_0\nabla f_{\delta}$ as the forcing in the governing equation. The function $g$ and the approximated delta function $f_{\delta}$ is given by
$$
g = e^{-\frac{(t-0.215)^2}{0.15}}, \quad f_{\delta} = \frac{1}{2\pi\sqrt{s_1 s_2}}\sum_{i=1}^3 e^{-\left( \frac{(x-x_i)^2}{2s_1} + \frac{(y-y_i)^2}{2s_2} \right)},
$$
where the parameters $s_1, s_2 = 0.5 h$,  half the grid-spacing in the physical domain, and $M_0 = 1000$. We apply the forcing at three different locations $(x_1, y_1) = (2.54796, -1.04916)$, $(x_2, y_2) = (8.4932, -1.04916)$ and $(x_3, y_3) =  (14.43844, -1.04916)$. We assume zero initial conditions for the displacement, velocity and the auxiliary variables. 
As the number of unknowns is relatively large, $\sim 3.5 \cdot 10^6$, we solve till the final time $T=10$. At the top boundary $(y=-0.44964)$, we use a traction-free boundary condition and on the left $(x=0)$, bottom $(y = -3.4972)$ and right $(x = 16.9864)$ boundaries, we use the characteristic  boundary condition. \rev{These boundary conditions take the same form as in \eqref{FourLayerBCtop}-\eqref{FourLayerBCbottom}. }

In Figure~\ref{fig:marmousi_snapshots}, we show the snapshots of the solution at four different times. We observe the scattering of elastic waves  due to the heterogeneous and discontinuous nature of the medium. At $t=0.5$, we observe that  the waves reach the top boundary and gets reflected due to the traction-free boundary condition. At $t=1.0$, we observe the effects of the interface as the waves reach the second layer. On the right-hand side of the domain, we see that the waves near the PML gets damped. From the snapshots at $t=1.5$ and $t=10$, we observe that the waves generated due to the three point-sources interact with each other and then subsequently reach the PML. We observe more clearly that the waves reaching the PML gets damped. In Figure~\ref{fig:norm-vs-t-marmousi}, we observe that $\|\mathbf{u}\|_{\mathbf{H}}$ decays monotonically with respect to $t$. \sw{This indicates the stability of the proposed PML model even when it is not covered by our stability analysis}. 

\begin{figure}
\centering
    \begin{subfigure}{\textwidth}
        \includegraphics[width=\textwidth]{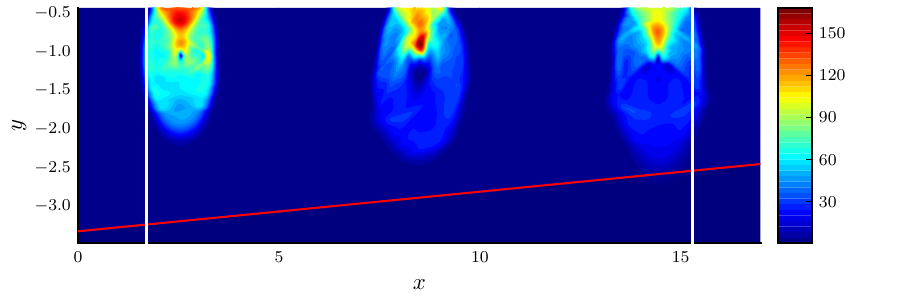}
    \end{subfigure}
    \begin{subfigure}{\textwidth}
        \includegraphics[width=\textwidth]{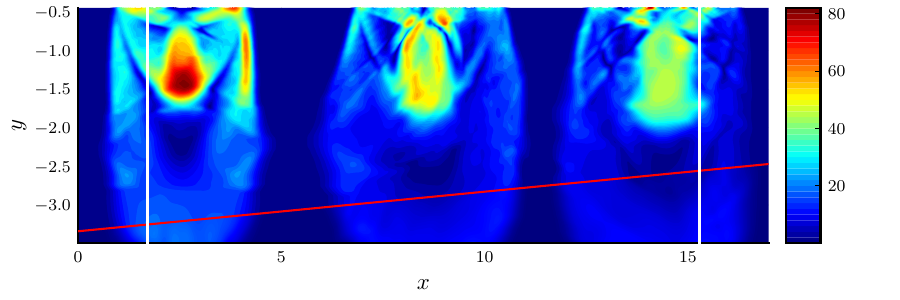}
    \end{subfigure}
    
    \begin{subfigure}{\textwidth}
        \includegraphics[width=\textwidth]{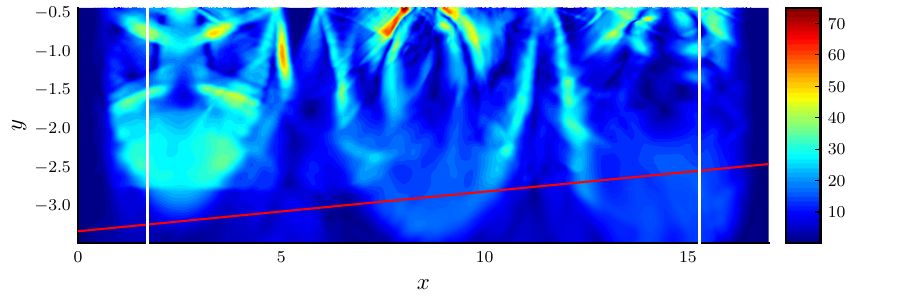}
    \end{subfigure}
    \begin{subfigure}{\textwidth}
        \includegraphics[width=\textwidth]{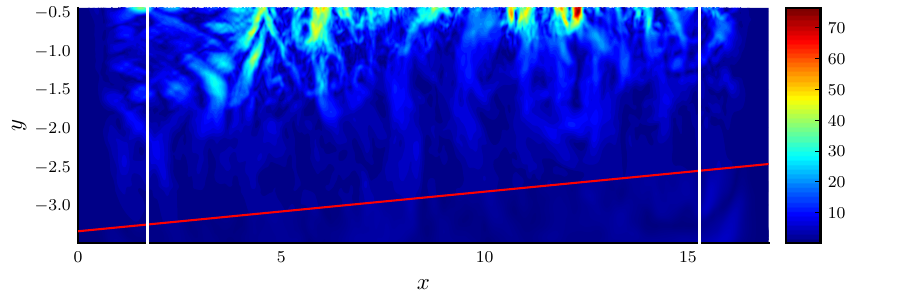}
    \end{subfigure}     


    \caption{(Top to bottom) Snapshots of the numerical solution  in a heterogeneous and discontinuous elastic solid defined by the \kd{Marmousi} model and truncated by PML model. The solutions are plotted at times $t = 0.5, 1, 1.5$ and $10$.}
    \label{fig:marmousi_snapshots}
\end{figure}

\begin{figure}
\centering
         \includegraphics[width=0.65\textwidth]{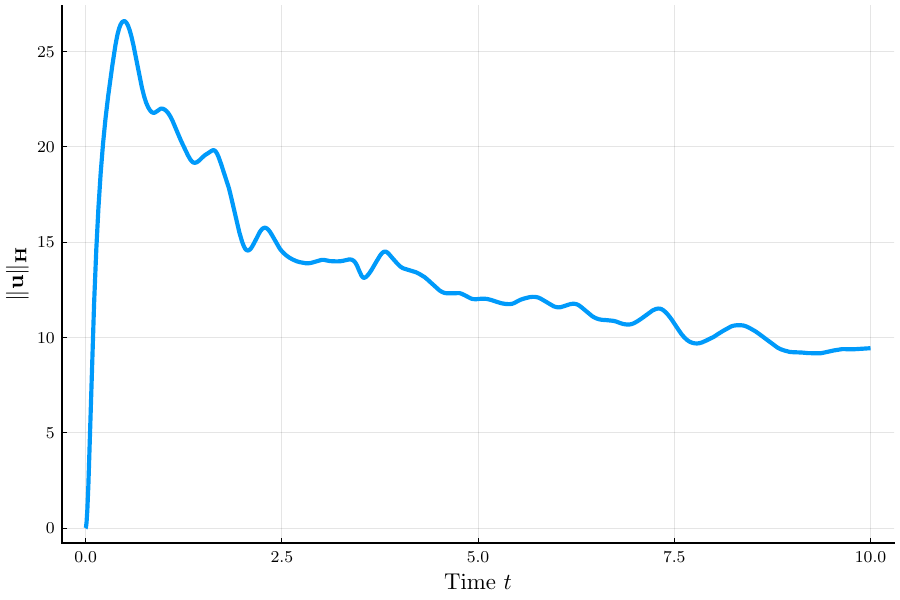}        
    \caption{The  quantity $\| \mathbf{u}\|_{\mathbf{H}}$ with $t$ for the Marmousi2 model.}
    \label{fig:norm-vs-t-marmousi}
\end{figure}

\section{Conclusion}\label{sec:conclusions}
We have analysed the stability of the PML for the elastic wave equation with piecewise constant material parameters and interface conditions at material interfaces. The elastic wave equation and  the interface conditions, without the PML, satisfy an energy estimate in physical space. Alternatively, a mode analysis can also be used to prove that exponentially growing modes are not supported by  the elastic wave equation subject to the interface conditions.  In particular, the normal mode analysis in Laplace space for interface waves gives  a boundary matrix $\mathcal{C}(s, k_x)$ of which the  determinant is a homogeneous function $\mathcal{F}(s, k_x)$ of $(s, k_x)$ and does not have any roots $s$ with a positive real part  $\Re{s} >0$ in the complex plane.  When the PML is present,  the energy method is in general not applicable but the normal mode analysis can be used to investigate the existence of exponentially growing modes in the PML. The normal mode analysis when applied to the PML in a discontinuous elastic medium  yields a similar boundary  matrix perturbed by the PML. Our analysis shows that if the PML IVP does not support growing modes, then the PML moves the roots of the determinant $\mathcal{F}(s, k_x)$ further into the stable complex plane. This proves that interface wave modes present at layered material interfaces in elastic solids are dissipated by the PML.
Our analysis builds upon the ideas presented in \cite{Duru2014K} and extends the stability results of  boundary waves (such as Rayleigh waves) on a half-plane elastic solid to   interface wave modes (such as Stoneley waves) transmitted into the PML at a planar interface separating two half-plane elastic solids.
We have presented numerical examples for both isotropic and anisotropic elastic solids verifying the analysis, and demonstrating that interface wave modes decay in the PML. Numerical examples using the Marmousi model \cite{marmousi2} demonstrates the utility of the PML and our numerical method for seismological applications.

\section*{Acknowledgement}
Part of the work was carried out during S. Wang’s research visit at Australian National University (ANU). The financial support from SVeFUM and ANU is greatly appreciated. B. Kalyanaraman was supported by the Kempe foundation (project number JCK22-0012). The authors thank Prof. Martin Almquist from Uppsala University for sharing the data for the Marmousi model.

\appendix

\section{ Perfectly Matched Layers in both spatial directions}
\label{app:1}
To derive the PML model in both spatial directions, we consider the Laplace transformed elastic wave equation~\eqref{lap1} in the transformed coordinate $(\tilde x, \tilde y)$, such that
\begin{align*}
    \frac{d \tilde{x}}{d x} = 1 + \frac{\sigma_x(x)}{\alpha + s} =: \Sx, \qquad \frac{d \tilde{y}}{d y} = 1 + \frac{\sigma_y(y)}{\alpha + s} =: \Sy.
\end{align*}
Here $\sigma_x(x) \ge 0$ and $\sigma_y(y) \ge 0$ are the damping functions and $\alpha \ge 0$ is the complex frequency shift. By introducing a smooth coordinate transform
\begin{align*}
    \frac{\partial }{\partial x} \to \frac{1}{\Sx}\frac{\partial}{\partial x}, \quad \frac{\partial }{\partial y} \to \frac{1}{\Sy}\frac{\partial}{\partial y},
\end{align*}
the PML model in the Laplace space becomes
\begin{equation}
s^2 \rho_i\mb{\widehat u}_i = \f{1}{\Sx}\f{\p}{\p x}\left(\f{1}{\Sx}A_i\f{\p\mb{\widehat u}_i}{\p x}\right)+\f{1}{\Sy}\f{\p}{\p y}\left(\f{1}{\Sy}B_i\f{\p\mb{\widehat u}_i}{\p y}\right)+\f{1}{\Sx\Sy}\f{\p}{\p x}\left(C_i\f{\p\mb{\widehat u}_i}{\p y}\right)+\f{1}{\Sx\Sy}\f{\p}{\p y}\left(C^T_i\f{\p\mb{\widehat u}_i}{\p x}\right),
\end{equation}
$i=1,2$, along with the interface conditions
\begin{equation}
    \mb{\widehat u}_1=\mb{\widehat u}_2,\quad B_1\f{1}{\Sy}\f{\p\mb{\widehat u}_1}{\p y}+C_1^T\f{1}{\Sx}\f{\p\mb{\widehat u}_1}{\p x}=B_2\f{1}{\Sy}\f{\p\mb{\widehat u}_2}{\p y}+C_2^T\f{1}{\Sx}\f{\p\mb{\widehat u}_2}{\p x}.
\end{equation}
Introducing the auxiliary variables 
\begin{align*}
    \hat{\mathbf{v}}_i = \frac{1}{\alpha + s + \sigma_x}\frac{\partial \hat{\mathbf{u}}_i}{\partial x},\quad     \hat{\mathbf{w}}_i = \frac{1}{\alpha + s + \sigma_y}\frac{\partial \hat{\mathbf{u}}_i}{\partial y},\quad \hat{\mathbf{q}}_i = \frac{\alpha}{\alpha + s}\hat{\mathbf{u}}_i, \quad \hat{\mathbf{r}}_i = \frac{\alpha}{\alpha+s}(\hat{\mathbf{u}}_i - \hat{\mathbf{q}}_i) 
\end{align*}
and inverting the Laplace transformed PDE, we obtain the system
\begin{subequations}
{\small
\begin{align}
  &\rho\left( \frac{\partial^2 {\mathbf{u}}_i}{\partial t^2} + (\sigma_x + \sigma_y)\frac{\partial {\mathbf{u}}_i}{\partial t} - \alpha\left(\sigma_x + \sigma_y
  \right)({\mathbf{u}}_i - {\mathbf{q}}_i)  + \sigma_x \sigma_y ({\mathbf{u}}_i - {\mathbf{q}}_i - {\mathbf{r}}_i) \right) =\nonumber \\
  & \qquad \frac{\partial }{\partial x}\left(A_i \frac{\partial {\mathbf{u}}_i}{\partial x} + (\sigma_y - \sigma_x) A_i {\mathbf{v}}_i + C_i \frac{\partial {\mathbf{u}}_i}{\partial y}\right) + \frac{\partial }{\partial y}\left(B \frac{\partial {\mathbf{u}}_i}{\partial y} + (\sigma_x - \sigma_y) B_i {\mathbf{w}}_i + C_i^T \frac{\partial {\mathbf{u}}_i}{\partial x}\right),\label{eq:first}\\
  &\qquad \qquad  \frac{\partial {\mathbf{v}}_i}{\partial t} = \frac{\partial {\mathbf{u}}_i}{\partial x} - (\alpha + \sigma_x){\mathbf{v}}_i,\\
  &\qquad \qquad\frac{\partial {\mathbf{w}}_i}{\partial t} = \frac{\partial {\mathbf{u}}_i}{\partial y} - (\alpha + \sigma_y){\mathbf{w}}_i,\\
  &\qquad \qquad\frac{\partial \mathbf{q}_i}{\partial t} = \alpha(\mathbf{u}_i - \mathbf{q}_i),\\
  &\qquad \qquad\frac{\partial \mathbf{r}_i}{\partial t} = \alpha(\mathbf{u}_i - \mathbf{q}_i - \mathbf{r}_i)\label{eq:extra}.
\end{align}\label{eq:x-y-pml}}%
\end{subequations}
along with the interface conditions
\begin{equation}\label{eq:x-y-pml-int}
    B_1 \frac{\partial {\mathbf{u}}_1}{\partial y} + (\sigma_x - \sigma_y) B_1 {\mathbf{w}}_1 + C_1^T \frac{\partial {\mathbf{u}}_1}{\partial x} = B_2 \frac{\partial {\mathbf{u}}_2}{\partial y} + (\sigma_x - \sigma_y) B_2 {\mathbf{w}}_2 + C_2^T \frac{\partial {\mathbf{u}}_2}{\partial x}.
\end{equation}
Setting $\sigma_y \equiv 0$ reduces \eqref{eq:x-y-pml}--\eqref{eq:x-y-pml-int} to the PML equations discussed in Section \ref{sec:pml_model}. We note that adding the PML along the $y$-direction introduces \eqref{eq:extra} coupled with \eqref{eq:first}.

\section{Elastic wave equation on curvilinear domains} \label{app:2}
Let us assume that the PML model \eqref{eq:PML_physical_space} is defined on curvilinear domains. We introduce a smooth transformation $(q(x,y),r(x,y)) \leftrightarrow (x(q,r),y(q,r))$, for each block, where a point $(q,r)$ in the reference domain is mapped to a point $(x,y)$ in the physical domain. We denote by $\mathbf{J}$ and $\mathbf{J}^{-1}$, the Jacobian matrix and its inverse, defined as
\begin{equation}
    \mathbf{J} = \begin{bmatrix}
                    \displaystyle
                    \frac{\partial x}{\partial q} & \displaystyle \frac{\partial y}{\partial q}\\\\
                    \displaystyle                    
                    \frac{\partial x}{\partial r} & \displaystyle \frac{\partial y}{\partial r}
                \end{bmatrix},\qquad
    \mathbf{J}^{-1} = \begin{bmatrix}
                    \displaystyle
                    \frac{\partial q}{\partial x} & \displaystyle \frac{\partial r}{\partial x}\\\\
                    \displaystyle                    
                    \frac{\partial q}{\partial y} & \displaystyle \frac{\partial r}{\partial y}
                \end{bmatrix}.\nonumber
\end{equation}
We apply the transformation
\begin{align*}
    \frac{\partial }{\partial x} = \frac{\partial }{\partial q}\left(\frac{\partial q}{\partial x}\right) + \frac{\partial }{\partial r}\left(\frac{\partial r}{\partial x}\right), \qquad \frac{\partial }{\partial y} = \frac{\partial }{\partial q}\left(\frac{\partial q}{\partial y}\right) + \frac{\partial }{\partial r}\left(\frac{\partial r}{\partial y}\right),
\end{align*}
and write the elastic wave equation in the reference coordinates as follows.
\begin{align*}
  &|\mathbf{J}|\rho_i\left( \frac{\partial^2 \mathbf{u}_i}{\partial t^2} + (\sigma_x + \sigma_y)\frac{\partial \mathbf{u}_i}{\partial t} - \sigma_x \alpha(\mathbf{u}_i - \mathbf{q}_i) + \sigma_y \alpha(\mathbf{u}_i - \mathbf{q}_i) \right) = \\ 
  & \qquad \qquad \frac{\partial }{\partial q}\left(\tilde{A}_i \frac{\partial \mathbf{u}_i}{\partial q} + \tilde{C}_i \frac{\partial \mathbf{u}_i}{\partial r} + \hat{A}_i \mathbf{v}_i\right) + \frac{\partial }{\partial r}\left(\tilde{B}_i \frac{\partial \mathbf{u}_i}{\partial r} + \tilde{C}_i^T \frac{\partial \mathbf{u}_i}{\partial q} + \hat{B}_i \mathbf{w}_i\right),\\
  &\qquad \frac{\partial \mathbf{v}_i}{\partial t} = \frac{\partial \mathbf{u}_i}{\partial q}\frac{\partial q}{\partial x}  + \frac{\partial \mathbf{u}_i}{\partial r} \frac{\partial r}{\partial x}  - (\alpha + \sigma_x)\mathbf{v}_i,\\
  &\qquad \frac{\partial \mathbf{w}_i}{\partial t} = \frac{\partial \mathbf{u}_i}{\partial q} \frac{\partial q}{\partial y} + \frac{\partial \mathbf{u}_i}{\partial r} \frac{\partial r}{\partial y} - (\alpha + \sigma_y)\mathbf{w}_i,\\
  &\qquad \frac{\partial \mathbf{q}_i}{\partial t} = \alpha(\mathbf{u}_i - \mathbf{q}_i),
\end{align*}
where 
\begin{align*}
   \begin{bmatrix}
      \tilde{A}_i & \tilde{C}_i\\   
      \tilde{C}_i^T & \tilde{B}_i
   \end{bmatrix} = 
    |\mathbf{J}| \begin{bmatrix}
            \mathbf{J}^{-1} & 0 \\
            0 & \mathbf{J}^{-1}
          \end{bmatrix}^T 
          \begin{bmatrix}
            A_i & C_i\\   
            C_i^T & B_i
          \end{bmatrix}&
          \begin{bmatrix}
            \mathbf{J}^{-1} & 0 \\
            0 & \mathbf{J}^{-1}
          \end{bmatrix},\quad
    \hat{A}_i = |\mathbf{J}|(\sigma_y - \sigma_x)\mathbf{J}^{-1}A_i,\\
    \hat{B}_i &= |\mathbf{J}|(\sigma_x - \sigma_y)\mathbf{J}^{-1}B_i.
\end{align*}
The stress tensor $\tau_i$ and the outward normal $\mathbf{n}$ on the boundaries can be written as
\begin{align}
    \tau_i = \begin{bmatrix}
                    \displaystyle
                    A_i \left(\frac{\partial \mathbf{u}_i}{\partial q}\frac{\partial q}{\partial x}  + \frac{\partial \mathbf{u}_i}{\partial r}\frac{\partial r}{\partial x}\right) + C_i \left(\frac{\partial \mathbf{u}_i}{\partial q}\frac{\partial q}{\partial y}  + \frac{\partial \mathbf{u}_i}{\partial r}\frac{\partial r}{\partial y}\right) + (\sigma_y    - \sigma_x)A_i \mathbf{v}_i\\\\
                    \displaystyle
                    C_i^T \left(\frac{\partial \mathbf{u}_i}{\partial q}\frac{\partial q}{\partial x}  + \frac{\partial \mathbf{u}_i}{\partial r}\frac{\partial r}{\partial x}\right) + B_i \left(\frac{\partial \mathbf{u}_i}{\partial q}\frac{\partial q}{\partial y}  + \frac{\partial \mathbf{u}_i}{\partial r}\frac{\partial r}{\partial y}\right) + (\sigma_x - \sigma_y)B_i\mathbf{w}_i
                  \end{bmatrix}, \quad \mathbf{n} = \frac{\mathbf{J}^{-1}\mathbf{\hat{n}}}{|\mathbf{J}^{-1}\mathbf{\hat{n}}|}\nonumber
\end{align}
where $\hat{\mathbf{n}}$ is the corresponding outward normal in the reference domain. Let $\mathbf{n}_{\Gamma,1} = -\mathbf{n}_{\Gamma,2} = \mathbf{n}_{\Gamma}$ denote the outward normal of the interface boundary $\Gamma$ in the physical domain. The interface condition enforcing the continuity of traction is as follows
$$
\tau_1 \cdot \mathbf{n} = \tau_2 \cdot \mathbf{n}.
$$
Without loss of generality, let us assume that the interface is mapped to the side $r=0$, where $\hat{\mathbf{n}}_1 = -\hat{\mathbf{n}}_2 = [0, -1]$. Now in terms of the reference coordinates, the traction on the interface $\Gamma$ satisfies
\begin{align}
\frac{1}{|\mathbf{J}||\mathbf{J}^{-1}{\mathbf{\hat{n}}_1|}}\left(\tilde{B}_1 \frac{\partial \mathbf{u}_1}{\partial r} + \tilde{C}_1^T \frac{\partial \mathbf{u}_1}{\partial q} + \hat{B}_1 \mathbf{w}_1\right) = \frac{1}{|\mathbf{J}||\mathbf{J}^{-1}\mathbf{\hat{n}}_2|}\left(\tilde{B}_2 \frac{\partial \mathbf{u}_2}{\partial r} + \tilde{C}_2^T \frac{\partial \mathbf{u}_2}{\partial q} + \hat{B}_2 \mathbf{w}_2\right).\nonumber
\end{align}
The displacement continuity remains unchanged and can be written as 
$$
\mathbf{u}_1 = \mathbf{u}_2 \quad \text{on} \quad \Gamma.
$$
Now we have a formulation where the derivatives are in terms of the reference coordinates $q$ and $r$. We can now discretise the PDE using the summation-by-parts operators defined on a unit square and solve for the displacements.

\bibliographystyle{plain}
\bibliography{Siyang_References}

\end{document}